\documentclass{siamltex}
\usepackage[T1]{fontenc}
\usepackage[utf8]{inputenc}
\usepackage{lmodern}
\usepackage{amssymb}
\usepackage{amsmath}
\usepackage{amsfonts}
\usepackage{sw20siam}
\usepackage{graphicx}
\usepackage{epsfig}
\usepackage{multicol}
\usepackage{epstopdf}
\usepackage{subcaption}
\usepackage{float}

\newtheorem {remark}[theorem]{Remark}

\input {tcilatex}
\begin{document}

\title{Analysis of variable-step/non-autonomous artificial compression
methods}
\author{Robin Ming Chen \thanks{%
Department of Mathematics, University of Pittsburgh, Pittsburgh, PA 15260,
USA, mingchen@pitt.edu; partially supported by NSF Grant DMS 1613375.} \and William Layton \thanks{%
Department of Mathematics, University of Pittsburgh, Pittsburgh, PA 15260,
USA, wjl@pitt.edu; The research herein was partially supported by NSF\
grants DMS1522267, 1817542 and CBET 1609120.} \and Michael Mclaughlin 
\thanks{%
Department of Mathematics, University of Pittsburgh, Pittsburgh, PA 15260,
USA, mem266@pitt.edu; partially supported by NSF\ grants DMS1522267 and
1817542.}}
\date{11/23/2017}
\maketitle

\begin{abstract}
A standard artificial compression (AC) method for incompressible flow is 
\begin{gather*}
\frac{u_{n+1}^{\varepsilon }-u_{n}^{\varepsilon }}{k}+u_{n+1}^{\varepsilon
}\cdot \nabla u_{n+1}^{\varepsilon }+{\frac{1}{2}}u_{n+1}^{\varepsilon
}\nabla \cdot u_{n+1}^{\varepsilon }+\nabla p_{n+1}^{\varepsilon }-\nu
\Delta u_{n+1}^{\varepsilon }=f\text{ ,} \\
{\text{\ }}\varepsilon \frac{p_{n+1}^{\varepsilon }-p_{n}^{\varepsilon }}{k}%
+\nabla \cdot u_{n+1}^{\varepsilon }=0
\end{gather*}%
for, typically, $\varepsilon =k$ (timestep). It is fast, efficient and
stable with accuracy $O(\varepsilon +k)$. For adaptive (and thus variable)
timestep $k_{n}$ (and thus $\varepsilon =\varepsilon _{n}$) its long time
stability is unknown. For variable $k,\varepsilon $ this report shows how to
adapt a standard AC method to recover a provably stable method. For the
associated continuum AC model, we prove convergence of the $\varepsilon
=\varepsilon (t)$\ artificial compression model to a weak solution of the
incompressible Navier-Stokes equations as $\varepsilon =\varepsilon
(t)\rightarrow 0$. The analysis is based on space-time Strichartz estimates
for a non-autonomous acoustic equation. Variable $\varepsilon ,k$ numerical
tests in $2d$ and $3d$ are given for the new AC method.
\end{abstract}

\section{Introduction}

Of the many methods for predicting incompressible flow, artificial
compression (AC) methods, based on replacing $\nabla \cdot u=0$ by $%
\varepsilon p_{t}+\nabla \cdot u=0$ ($0<\varepsilon $ small) and advancing
the pressure explicitly in time, are among the most efficient. These methods
also have a reputation for low time accuracy. Herein we study one source of
low accuracy, propose a resolution, give analytical support for the
corrected method and show some numerical comparisons of a common AC method
and its proposed correction. Consider the incompressible Navier-Stokes
equations in a $3d$ domain $\Omega ,$ here either a bounded open set or ${{%
\mathbb{R}}}^{3}$, 
\begin{equation}
\begin{cases}
& u_{t}+(u\cdot \nabla )u+\nabla p-\nu \Delta u=f(t,x) \\ 
& \nabla \cdot u=0,%
\end{cases}
\label{NS}
\end{equation}%
where $(t,x)\in \lbrack 0,T]\times \Omega $, $u\in {{\mathbb{R}}}^{3}$ is
the velocity, $p\in {{\mathbb{R}}}$ the pressure, $\nu $ the kinematic
viscosity, and $f\in {{\mathbb{R}}}^{3}$ the external force.

AC methods, e.g., \cite{GMS06}, \cite{P97}, \cite{DLM17}, are based on
approximating the solution of the slightly compressible equations 
\begin{equation}
\begin{cases}
& u_{t}^{\varepsilon }+(u^{\varepsilon }\cdot \nabla )u^{\varepsilon }+{%
\frac{1}{2}}(\nabla \cdot u^{\varepsilon })u^{\varepsilon }+\nabla
p^{\varepsilon }-\nu \Delta u^{\varepsilon }=f \\ 
& \varepsilon p_{t}^{\varepsilon }+\nabla \cdot u^{\varepsilon }=0,\text{
where }0<\varepsilon \text{ is small.}%
\end{cases}
\label{AC}
\end{equation}%
Here $u^{\varepsilon }$ is the approximate velocity, $p^{\varepsilon }$ is
the approximate pressure and the nonlinearity has been explicitly
skew-symmetrized. (This is a common formulation but not the only one,
Section 1.1). Time accuracy is obtained\footnote{%
The separate issue of pressure initialization, not addressed here, also
exists. We do note that for internal flows pressure data is often more
reliable than velocity data.} by either using explicit time discretization
methods and small time steps for short time simulations, using high order
methods with moderate timesteps for longer time simulations or by adding
time adaptivity to a low or high order implicit method. The first is not
considered herein. The second leads to highly ill-conditioned linear systems%
\footnote{%
For a $4$th order time discretization, $\varepsilon =O(k^{4})$ is necessary
to retain accuracy. This leads to a viscous term \ $-\nu \Delta
u_{n+1}^{\varepsilon }-k^{-3}\nabla \nabla \cdot u_{n+1}^{\varepsilon }$ and
a linear system to be solved at each timestep with condition number $%
\mathcal{O}(timestep^{-2}\times spacemesh^{-2})$.}. (Shen \cite{S96} also
suggests that an accuracy barrier exists in AC methods.) The third, considered
herein, has the possibility to both increase efficiency and provide time
accuracy. To our knowledge, the defect correction based scheme of Guermond
and Minev \cite{GM18} is the only previous work in this direction.

To fix ideas, suppress the space discretization and consider a commonly used
fully-implicit time discretization 
\begin{gather*}
\frac{u_{n+1}^{\varepsilon }-u_{n}^{\varepsilon }}{k}+(u_{n+1}^{\varepsilon
}\cdot \nabla )u_{n+1}^{\varepsilon }+{\frac{1}{2}}(\nabla \cdot
u_{n+1}^{\varepsilon })u_{n+1}^{\varepsilon }+\nabla p_{n+1}^{\varepsilon
}-\nu \Delta u_{n+1}^{\varepsilon }=f(t_{n+1}), \\
{\text{ \ \ \ }}\varepsilon \frac{p_{n+1}^{\varepsilon }-p_{n}^{\varepsilon }%
}{k}+\nabla \cdot u_{n+1}^{\varepsilon }=0.
\end{gather*}%
For other time discretizations see, e.g., \cite{K86}, \cite{OA10}, \cite%
{DLM17}, \cite{YBC16}. Here $k$ is the timestep, $t_{n}=nk,$ $%
u_{n}^{\varepsilon }$, $p_{n}^{\varepsilon }$ are approximations to the
velocity and pressure at $t=t_{n}$. Since $\nabla p_{n+1}^{\varepsilon }=$ $%
\nabla p_{n}^{\varepsilon }-(k/\varepsilon )\nabla \nabla \cdot
u_{n+1}^{\varepsilon }$, this uncouples into%
\begin{gather}
\frac{u_{n+1}^{\varepsilon }-u_{n}^{\varepsilon }}{k}+(u_{n+1}^{\varepsilon
}\cdot \nabla )u_{n+1}^{\varepsilon }+{\frac{1}{2}}(\nabla \cdot
u_{n+1}^{\varepsilon })u_{n+1}^{\varepsilon }+\nabla p_{n}^{\varepsilon }-%
\frac{k}{\varepsilon }\nabla \nabla \cdot u_{n+1}^{\varepsilon }  \notag \\
-\nu \Delta u_{n+1}^{\varepsilon }=-\nabla p_{n}^{\varepsilon }+f_{n+1},
\label{eq:StandardACmethod} \\
{\text{ {then given }}u_{n+1}^{\varepsilon }\text{{: \ }}}%
p_{n+1}^{\varepsilon }=p_{n}^{\varepsilon }-(k/\varepsilon )\nabla \cdot
u_{n+1}^{\varepsilon }.  \notag
\end{gather}%
This method is unconditionally, nonlinearly, long time stable, e.g., \cite%
{GMS06}, \cite{GM15}. It has consistency error $\mathcal{O}(k+\varepsilon )$
and thus determines $\varepsilon $ balancing errors by $\varepsilon =k$.
Time adaptivity means decreasing or increasing the time step according to
solution activity, \cite{GS00}. Given the $\mathcal{O}(k+\varepsilon )$
consistency error, this means varying both $k=k_{n}$ and $\varepsilon
=\varepsilon _{n}$. To our knowledge, no long time stability analysis of
this method with variable $k=k_{n}$ and $\varepsilon =\varepsilon _{n}$\ is
known or even possible at present, Section 2. Peculiar solution behavior
seen in an adaptive simulation thus cannot be ascribed to either a flow
phenomenon or to an anomaly created by the numerical method. This is the
problem we address herein for the time discretized AC method (with $%
\varepsilon =\varepsilon _{n}$) and for the associated continuum AC model
(with $\varepsilon =\varepsilon (t)$).

In Section 2\ we first show that the standard AC method, (\ref%
{eq:StandardACmethod}) above, is 0-stable for variable $\varepsilon ,k$
provided $\varepsilon ,k$ are slowly varying. 0-stability allows
non-catastrophic exponential growth. Thus, (\ref{eq:StandardACmethod})
suffices for short time simulations with nearly constant timesteps. The long
time stability of (\ref{eq:StandardACmethod}) with variable-$\varepsilon ,k$
is analyzed in Section 2 as well. Some preliminary conclusions are presented
but then complete resolution of instability or stability is an open problem
for the standard method.

Section 2 presents a stable extension of AC methods to variable $\varepsilon
,k$, one central contribution of this report. The proposed method is%
\begin{gather}
\frac{u_{n+1}^{\varepsilon }-u_{n}^{\varepsilon }}{k_{n+1}}+\nabla
p_{n+1}^{\varepsilon }-\nu \Delta u_{n+1}^{\varepsilon
}+(u_{n+1}^{\varepsilon }\cdot \nabla )u_{n+1}^{\varepsilon }+{\frac{1}{2}}%
(\nabla \cdot u_{n+1}^{\varepsilon })u_{n+1}^{\varepsilon }=f_{n+1},  \notag
\\
\frac{1}{2}\frac{\varepsilon _{n+1}p_{n+1}^{\varepsilon }-\varepsilon
_{n}p_{n}^{\varepsilon }}{k_{n+1}}+\frac{\varepsilon _{n}}{2}\frac{%
p_{n+1}^{\varepsilon }-p_{n}^{\varepsilon }}{k_{n+1}}+\nabla \cdot
u_{n+1}^{\varepsilon }=0.  \label{eq:variableStepAC}
\end{gather}%
This method reduces to the standard AC method (\ref{eq:StandardACmethod})
for constant $\varepsilon ,k$. Section 2 shows that the new method (\ref%
{eq:variableStepAC}) is unconditionally, nonlinearly, long time stable
without assumptions on $\varepsilon _{n},k_{n}$, Theorem 2.3. In numerical
tests of (\ref{eq:variableStepAC}) in Section 5, the new method works well
(as expected) when $k_{n+1}=\varepsilon _{n+1}$ is picked self adaptively to
ensure $||\nabla \cdot u||$ is below a present tolerance. It also performs
well in tests where $k_{n}=\varepsilon _{n}$ is pre-chosen to try to break
the method's stability or physical fidelity by increasing or fluctuating $%
\varepsilon ,k$.

In support, we give an analysis of the physical fidelity of the
non-autonomous continuum model associated with (\ref{eq:variableStepAC}): 
\begin{equation}
\begin{array}{c}
\partial _{t}u^{\varepsilon }+(u^{\varepsilon }\cdot \nabla )u^{\varepsilon
}+{\frac{1}{2}}(\nabla \cdot u^{\varepsilon })u^{\varepsilon }+\nabla
p^{\varepsilon }-\nu \Delta u^{\varepsilon }=f, \\ 
\partial _{t}(\varepsilon (t)p^{\varepsilon })-{\frac{1}{2}}\varepsilon
_{t}(t)p^{\varepsilon }+\nabla \cdot u^{\varepsilon }=0.%
\end{array}
\label{ACm}
\end{equation}%
Sections 3 and 4 address the question:\textit{\ Under what conditions on }$%
\varepsilon (t)$\textit{\ do solutions to the new AC model (1.3) converge to
weak solutions of the incompressible NSE as }$\varepsilon \rightarrow 0$%
\textit{?} Convergence (modulo a subsequence) is proven under the assumption
on the fluctuation $\varepsilon _{t}(t)$ and the second variation $%
\varepsilon _{tt}(t)$ that 
\begin{equation}
\varepsilon (t)\leq C\epsilon \rightarrow 0,\quad \lim_{\epsilon \rightarrow
0}{\frac{\varepsilon _{t}(t)}{\varepsilon (t)}}=0,\quad \lim_{\epsilon
\rightarrow 0}{\frac{\varepsilon _{tt}(t)}{\varepsilon (t)}}=0,\quad {{\text{
for }}}t\in \lbrack 0,T].  \label{condonepsilon}
\end{equation}%
This extension of model convergence to the non-autonomous system is a second
central contribution herein. In self-adaptive simulations based on (\ref%
{eq:variableStepAC}), this condition requires smooth adjustment of timesteps
and precludes a common strategy of timestep halving or doubling. A similar
smoothness condition on $\varepsilon _{t}(t)$ recently arose in stability
analysis of other variable timestep methods in \cite{SFR18}. Weakening the
condition (\ref{condonepsilon}) on $\varepsilon (t)$\ (which we conjecture
is possible) is an important open problem.

\subsection{Related work}

Artificial compression (AC) methods were introduced by Chorin \cite%
{chorin1968numerical,chorin1969convergence}, Oskolkov \cite%
{oskolkov1971aquasi} and Temam \cite%
{temam1969approximation1,temam1969approximation2}. For constant (not time
variable) $\varepsilon ,$ convergence of the AC approximation \eqref {AC} to
a weak solution of the NSE \eqref {NS} as $\varepsilon \rightarrow 0$ has
been proven for bounded 2d domains Temam \cite%
{temam1969approximation1,temam1969approximation2,temam2001book} (using the
method of fractional derivatives of Lions \cite{jllions1959}).
Donatelli-Marcati \cite{donatelli2006whole,donatelli2010exterior} extended $%
\varepsilon \rightarrow 0$\ convergence to the case of the 3d whole space
and exterior domains and in \cite{donatelli2006whole} by using the
dispersive structure of the acoustic pressure equation. There is also a
growing literature establishing convergence of discretizations of AC models
to NSE solutions including \cite{GM15}, \cite{GM17}, \cite{JL04}, \cite{K02}%
, \cite{S92a}.

\subsection{Analytical difficulties of the $\protect\varepsilon %
(t)\rightarrow 0$ limit}

For $\Omega ={{\mathbb{R}}}^{3}$, one difficulty in establishing convergence
is the estimate for acoustic pressure waves. From the acoustic pressure wave
equation \eqref {PWm} for the new model (\ref{ACm}), the pressure wave speed
is $\mathcal{O}(1/\varepsilon )$, suggesting only weak convergence of the
velocity $u^{\varepsilon }$. Strong convergence of $u^{\varepsilon }$ thus
hinges upon the dispersive behavior of these waves at infinity. In the case
when $\varepsilon $ is constant, the classical Strichartz type estimates 
\cite{GV95,KT98,S77} together with a refined bilinear estimate \cite%
{KM93,S95} of the three-dimensional inhomogeneous wave equations can be
directly applied to infer sufficient control of the pressure waves. However,
when $\varepsilon =\varepsilon (t)$, the resulting acoustic equation is 
\textit{non-autonomous}. There are still results on the space-time
Strichartz estimates for variable-coefficient wave equations at our
disposal, cf. Theorem \ref{thm_strichartz}. However the refined bilinear
estimates do not seem to be available, since these estimates are based on
the explicit structure of the Kirchhoff's formula for the classical wave
operator. To overcome this difficulty, we further introduce a scale change
in the time variable so that the principal part of the resulting pressure
wave equation becomes the classical wave operator. The lower order terms
have coefficients depending on the fluctuation and second variation of $%
\varepsilon $, and can be thought of as the forcing terms. That is the
motivation of the assumption \eqref {condonepsilon} under which the
coefficients of the lower order terms can be made small, $o(1)$, and hence
can be absorbed into the left-hand side of the Strichartz estimates. This
allows us to obtain the refined bilinear estimates, and therefore establish
the desired dispersive estimates for the pressure. Please refer to Section %
\ref{subsec_pressure} for more details.

\subsection{Other AC formulations}

Generally AC methods skew-symmetrize the nonlinearity and include a term $%
\varepsilon p_{t}$ that uncouples pressure and velocity and lets the
pressure be explicitly advanced in time. There are several choices for
the first and several for the second. A few alternate possibilities are
described next and combinations of these are certainly possible.

Motivated by the equations of hyposonic flow \cite{Z06}, the material
derivative can be used for the artificial compression term, e.g., \cite%
{oskolkov1971aquasi}, 
\begin{equation}
\begin{cases}
& \partial _{t}u^{\varepsilon }+(u^{\varepsilon }\cdot \nabla
)u^{\varepsilon }+{\frac{1}{2}}(\nabla \cdot u^{\varepsilon })u^{\varepsilon
}+\nabla p^{\varepsilon }-\nu \Delta u^{\varepsilon }=f, \\ 
& \varepsilon \left( \partial _{t}p^{\varepsilon }+u^{\varepsilon }\cdot
\nabla p^{\varepsilon }\right) +\nabla \cdot u^{\varepsilon }=0.%
\end{cases}%
\end{equation}%
Numerical dissipation can be incorporated into the pressure equation, e.g. 
\cite{K02}, as in 
\begin{equation}
\begin{cases}
& \partial _{t}u^{\varepsilon }+(u^{\varepsilon }\cdot \nabla
)u^{\varepsilon }+{\frac{1}{2}}(\nabla \cdot u^{\varepsilon })u^{\varepsilon
}+\nabla p^{\varepsilon }-\nu \Delta u^{\varepsilon }=f \\ 
& \varepsilon \left( \partial _{t}p^{\varepsilon }+p^{\varepsilon }\right)
+\nabla \cdot u^{\varepsilon }=0.%
\end{cases}%
\end{equation}%
A dispersive regularization has been included in the momentum equation in 
\cite{DLM17}, 
\begin{equation}
\begin{cases}
& \partial _{t}\left( u^{\varepsilon }-\frac{1}{\varepsilon }\nabla \nabla
\cdot u^{\varepsilon }\right) +(u^{\varepsilon }\cdot \nabla )u^{\varepsilon
}+{\frac{1}{2}}(\nabla \cdot u^{\varepsilon })u^{\varepsilon }+\nabla
p^{\varepsilon }-\nu \Delta u^{\varepsilon }=f, \\ 
& \varepsilon \partial _{t}p^{\varepsilon }+\nabla \cdot u^{\varepsilon }=0.%
\end{cases}%
\end{equation}%
The nonlinearity can be skew symmetrized in various ways, replacing $u\cdot
\nabla u$\ in the NSE by one of the following 
\begin{equation*}
\begin{array}{ccc}
\text{{standard skew-symmetrization}} & : & (u^{\varepsilon }\cdot \nabla
)u^{\varepsilon }+{\frac{1}{2}}(\nabla \cdot u^{\varepsilon })u^{\varepsilon
} \\ 
{\text{Rotational form}} & : & \left( \nabla \times u^{\varepsilon }\right)
\times u^{\varepsilon } \\ 
{\text{EMA form \cite{CHOR17}}} & : & (\nabla u^{\varepsilon }+\left( \nabla
u^{\varepsilon }\right) ^{T})u^{\varepsilon }+(\nabla \cdot u^{\varepsilon
})u^{\varepsilon }%
\end{array}%
\end{equation*}%
The penalty model (not studied herein) where $\nabla \cdot u=0$ is replaced
by $\nabla \cdot u^{\varepsilon }=-\varepsilon p^{\varepsilon }$, is
sometimes also viewed as an artificial compression model, \cite{P97}.

\section{Stability of variable-$\protect\varepsilon $ AC methods}

We begin by considering variable $\varepsilon $ stability of the standard
method%
\begin{gather}
\frac{u_{n+1}^{\varepsilon }-u_{n}^{\varepsilon }}{k_{n+1}}+\nabla
p_{n+1}^{\varepsilon }-\nu \Delta u_{n+1}^{\varepsilon
}+(u_{n+1}^{\varepsilon }\cdot \nabla )u_{n+1}^{\varepsilon }+{\frac{1}{2}}%
(\nabla \cdot u_{n+1}^{\varepsilon })u_{n+1}^{\varepsilon }=f_{n+1},  \notag
\\
{\text{ \ \ \ }}\varepsilon _{n+1}\frac{p_{n+1}^{\varepsilon
}-p_{n}^{\varepsilon }}{k_{n+1}}+\nabla \cdot u_{n+1}^{\varepsilon }=0,\text{
}  \label{eq:StandardMethod} \\
\text{subject to initial and boundary conditions:}  \notag \\
u_{0}^{\varepsilon }(x)=u_{0}(x), \quad p_{0}^{\varepsilon }(x)=p_{0}(x),\text{ in }%
\Omega \text{ ,}  \notag \\
u_{n}^{\varepsilon }=0\text{ on }\partial \Omega \text{ for }t>0\text{.} 
\notag
\end{gather}%
We first prove 0-stability, namely that $u_{n}^{\varepsilon }$ can grow no
faster than exponential, when $\varepsilon _{n}$ is slowly varying. The case
when $f\equiv 0$ is clearest since then any energy growth is then incorrect.

\begin{theorem}
For the standard method (\ref{eq:StandardMethod}), let%
\begin{equation*}
f_{n}=0\text{ for all }n
\end{equation*}%
and suppose%
\begin{equation*}
\left\vert \frac{\varepsilon _{n+1}-\varepsilon _{n}}{k_{n}}\right\vert \leq
\beta \varepsilon _{n}\text{ for some }\beta \text{ for all }n\text{.}
\end{equation*}%
Then%
\begin{eqnarray*}
\frac{1}{2}\int \left[ |u_{n}^{\varepsilon }|^{2}+\varepsilon _{n}\left(
p_{n}^{\varepsilon }\right) ^{2}\right] dx &\leq &\left( \Pi
_{j=1}^{n-1}(1+k_{j}\beta )\right) \frac{1}{2}\int \left[ |u_{0}|^{2}+%
\varepsilon _{0}\left( p_{0}\right) ^{2}\right] dx \\
&\leq &e^{\beta t_{n}}\frac{1}{2}\int \left[ |u_{0}|^{2}+\varepsilon
_{0}\left( p_{0}\right) ^{2}\right] dx.
\end{eqnarray*}
\end{theorem}

\begin{proof}
Take an inner product of the first equation with $k_{n+1}u_{n+1}^{%
\varepsilon }$, the second with $k_{n+1}p_{n+1}^{\varepsilon }$, integrate
over $\Omega $, integrate by parts, use skew symmetry and add. This yields,
by the polarization identity\footnote{%
Alternately, the algebraic identity $vw=\frac{1}{2}v^{2}+\frac{1}{2}w^{2}-%
\frac{1}{2}(v-w)^{2}.$}, 
\begin{gather*}
\frac{1}{2}\int |u_{n+1}^{\varepsilon }|^{2}-|u_{n}^{\varepsilon
}|^{2}+|u_{n+1}^{\varepsilon }-u_{n}^{\varepsilon }|^{2}dx+ \\
\frac{1}{2}\int \varepsilon _{n+1}\left( p_{n+1}^{\varepsilon }\right)
^{2}-\varepsilon _{n+1}\left( p_{n}^{\varepsilon }\right) ^{2}+\frac{%
\varepsilon _{n+1}}{2}|p_{n+1}^{\varepsilon }-p_{n}^{\varepsilon
}|^{2}dx+k_{n+1}\int \nu |\nabla u_{n+1}^{\varepsilon }|^{2}dx=0.
\end{gather*}%
The pressure terms do not collapse into a telescoping sum upon adding due to
the variability of $\varepsilon $. Thus we correct for this effect,
rearrange and adjust appropriately to yield%
\begin{gather*}
\frac{1}{2}\int |u_{n+1}^{\varepsilon }|^{2}+\varepsilon _{n+1}\left(
p_{n+1}^{\varepsilon }\right) ^{2}dx-\frac{1}{2}\int |u_{n}^{\varepsilon
}|^{2}dx+\varepsilon _{n}\left( p_{n}^{\varepsilon }\right) ^{2}dx+ \\
+\int \frac{1}{2}|u_{n+1}^{\varepsilon }-u_{n}^{\varepsilon }|^{2}+\frac{%
\varepsilon _{n+1}}{2}|p_{n+1}^{\varepsilon }-p_{n}^{\varepsilon }|^{2}dx+ \\
+k_{n+1}\int \nu |\nabla u_{n+1}^{\varepsilon }|^{2}dx=\frac{1}{2}\int
\left( \varepsilon _{n+1}-\varepsilon _{n}\right) \left( p_{n}^{\varepsilon
}\right) ^{2}dx.
\end{gather*}%
Note that $\int \left( \varepsilon _{n+1}-\varepsilon _{n}\right) \left(
p_{n}^{\varepsilon }\right) ^{2}dx\leq k_{n}\beta \int \varepsilon
_{n}\left( p_{n}^{\varepsilon }\right) ^{2}dx$. Dropping the (non-negative)
dissipation terms we have%
\begin{equation*}
y_{n+1}-y_{n}\leq k_{n}\beta y_{n}\text{ where }y_{n}:=\frac{1}{2}\int
|u_{n}^{\varepsilon }|^{2}dx+\varepsilon _{n}\left( p_{n}^{\varepsilon
}\right) ^{2}dx,
\end{equation*}%
from which the first result follows immediately. For the second inequality,
note that since $1+k_{j}\beta \leq e^{k_{j}\beta }$ we have%
\begin{equation*}
\Pi _{j=1}^{n-1}(1+k_{j}\beta )\leq \Pi _{j=1}^{n-1}e^{k_{j}\beta }=e^{\beta 
\left[ \sum_{j=1}^{n-1}k_{j}\right] }=e^{\beta t_{n}}
\end{equation*}
\end{proof}

Since $\beta $ (by assumption) is independent of the timestep $k$, this
implies 0-stability. For short time simulations, 0-stability suffices, but
is insufficient for simulations over longer time intervals. The assumption
that \ $\varepsilon $ (and thus also the timestep) is slowly varying:%
\begin{equation*}
\left\vert \frac{\varepsilon _{n+1}-\varepsilon _{n}}{k_{n}}\right\vert \leq
\beta \varepsilon _{n}
\end{equation*}%
precludes the common adaptive strategy of timestep halving and doubling. For
example, suppose%
\begin{eqnarray*}
\varepsilon _{n+1} &=&2\varepsilon _{n}\text{ and }k_{n+1}=2k_{n}\text{ then}
\\
\frac{\varepsilon _{n+1}-\varepsilon _{n}}{k_{n}\varepsilon _{n}} &=&\frac{%
2\varepsilon _{n}-\varepsilon _{n}}{k_{n}\varepsilon _{n}}=\frac{1}{k_{n}}%
\rightarrow \infty \text{ as }k\rightarrow 0.
\end{eqnarray*}

\subsection{The corrected, variable-$\protect\varepsilon $ AC method}

The above proof indicates that the problem arises from the fact that the
discrete $\varepsilon p_{t}$\ term is not a time difference when multiplied
by $p$. In all cases, the standard method obeys the discrete energy law%
\begin{gather}
\frac{1}{2}\int |u_{n+1}^{\varepsilon }|^{2}+\varepsilon _{n+1}\left(
p_{n+1}^{\varepsilon }\right) ^{2}dx-\frac{1}{2}\int |u_{n}^{\varepsilon
}|^{2}dx+\varepsilon _{n}\left( p_{n}^{\varepsilon }\right) ^{2}dx+
\label{eq:KElawStandardAC} \\
+\int \frac{1}{2}|u_{n+1}^{\varepsilon }-u_{n}^{\varepsilon }|^{2}+\frac{%
\varepsilon _{n+1}}{2}|p_{n+1}^{\varepsilon }-p_{n}^{\varepsilon
}|^{2}+k_{n+1}\nu |\nabla u_{n+1}^{\varepsilon }|^{2}dx=  \notag \\
=\frac{1}{2}\int \left( \varepsilon _{n+1}-\varepsilon _{n}\right) \left(
p_{n}^{\varepsilon }\right) ^{2}dx.
\end{gather}%
Since the variable-$\varepsilon $ term $(\varepsilon _{n}-\varepsilon
_{n+1})(p_{n})^{2}$ has two signs, depending only on whether the timestep is
increasing or decreasing, it can either dissipate energy or input energy.
The sign of the RHS shows that if:

\begin{itemize}
\item $k_{n}$ is decreasing the effect of changing the timestep is \textit{%
dissipative}, while if
\end{itemize}

\begin{itemize}
\item $k_{n}$ is increasing the effect of changing the timestep \textit{%
inputs energy} into the approximate solution.
\end{itemize}

In the second case, if the term $(\varepsilon _{n}-\varepsilon
_{n+1})|p_{n}|^{2}$ dominates in the aggregate the other dissipative terms
non-physical energy growth may be possible. However, we stress that we have
neither a proof of long time stability of the variable-$\varepsilon $
standard method nor a convincing example of instability. Resolving this is
an open problem discussed in the next sub-section.

The \ practical question is how to adapt the AC method to variable-$%
\varepsilon $ so as to ensure long time stability. After testing a few
natural alternatives we propose the new AC method 
\begin{eqnarray}
\frac{u_{n+1}^{\varepsilon }-u_{n}^{\varepsilon }}{k_{n+1}}%
+(u_{n+1}^{\varepsilon }\cdot \nabla )u_{n+1}^{\varepsilon }+{\frac{1}{2}}%
(\nabla \cdot u_{n+1}^{\varepsilon })u_{n+1}^{\varepsilon }-\nu \Delta
u_{n+1}^{\varepsilon }+\nabla p_{n+1}^{\varepsilon } &=&f_{n+1},  \notag \\
\frac{1}{2}\frac{\varepsilon _{n+1}p_{n+1}^{\varepsilon }-\varepsilon
_{n}p_{n}^{\varepsilon }}{k_{n+1}}+\frac{\varepsilon _{n}}{2}\frac{%
p_{n+1}^{\varepsilon }-p_{n}^{\varepsilon }}{k_{n+1}}+\nabla \cdot
u_{n+1}^{\varepsilon } &=&0.  \label{eq:NewACMethod}
\end{eqnarray}%
When $\varepsilon $ is constant the new method (\ref{eq:NewACMethod})
reduces to the standard method (\ref{eq:StandardACmethod}).

\begin{remark}[Higher Order Methods]
If a higher order time discretization such as BDF2 is desired, the
modification required is to use the higher order discretization for the
momentum equation, the same modification of the continuity equation and
select $\varepsilon _{n}=k_{n}^{\text{method order}}$ to preserve higher
order consistency error. For example, for variable step BDF2, let $\tau
=k_{n+1}/k_{n}.$ Then we have 
\begin{gather*}
\frac{\frac{2\tau +1}{\tau +1}u_{n+1}^{\varepsilon }-(\tau
+1)u_{n}^{\varepsilon }+\frac{\tau ^{2}}{\tau +1}u_{n-1}^{\varepsilon }}{%
k_{n+1}}+(u_{n+1}^{\varepsilon }\cdot \nabla )u_{n+1}^{\varepsilon }+{\frac{1%
}{2}}(\nabla \cdot u_{n+1}^{\varepsilon })u_{n+1}^{\varepsilon }\\
+\nabla
p_{n+1}^{\varepsilon }-\nu \Delta u_{n+1}^{\varepsilon }=f(t_{n+1}), \\
{\text{ \ }}\frac{1}{2}\frac{\varepsilon _{n+1}p_{n+1}^{\varepsilon
}-\varepsilon _{n}p_{n}^{\varepsilon }}{k_{n+1}}+\frac{\varepsilon _{n}}{2}%
\frac{p_{n+1}^{\varepsilon }-p_{n}^{\varepsilon }}{k_{n+1}}+\nabla \cdot
u_{n+1}^{\varepsilon }=0\text{ with }\varepsilon _{n+1}=k_{n+1}^{2}\text{.}
\end{gather*}%
This is easily proven $A-$stable for constant timesteps. Since BDF2 is not $%
A-$stable for increasing timesteps, the above would also not be expected to
be more than $0-$stable for increasing timesteps.
\end{remark}

\begin{theorem}
The variable-$\varepsilon ,k$ method (\ref{eq:NewACMethod}) is
unconditionally, long time stable. For any $N>0$ the energy equality holds: 
\begin{gather*}
\frac{1}{2}\int |u_{N}^{\varepsilon }|^{2}+\varepsilon
_{N}|p_{N}^{\varepsilon }|^{2}dx+ \\
+\sum_{n=0}^{N-1}\left[ \int \frac{1}{2}|u_{n+1}^{\varepsilon
}-u_{n}^{\varepsilon }|^{2}+\frac{\varepsilon _{n}}{2}(p_{n+1}^{\varepsilon
}-p_{n}^{\varepsilon })^{2}+k_{n+1}\nu |\nabla u_{n+1}^{\varepsilon }|^{2}dx%
\right] = \\
=\frac{1}{2}\int |u_{0}|^{2}+\varepsilon
_{0}|p_{0}|^{2}dx+\sum_{n=0}^{N-1}k_{n+1}\int u_{n+1}^{\varepsilon }\cdot
f_{n+1}dx
\end{gather*}%
and the stability bound holds:%
\begin{gather*}
\frac{1}{2}\int |u_{N}^{\varepsilon }|^{2}+\varepsilon
_{N}|p_{N}^{\varepsilon }|^{2}dx+ \\
+\sum_{n=0}^{N-1}\left[ \int \frac{1}{2}|u_{n+1}^{\varepsilon
}-u_{n}^{\varepsilon }|^{2}+\frac{\varepsilon _{n}}{2}(p_{n+1}^{\varepsilon
}-p_{n}^{\varepsilon })^{2}+k_{n+1}\frac{\nu }{2}|\nabla
u_{n+1}^{\varepsilon }|^{2}dx\right]  \\
=\frac{1}{2}\int |u_{0}|^{2}+\varepsilon
_{0}|p_{0}|^{2}dx+\sum_{n=0}^{N-1}k_{n+1}\frac{1}{2\nu }||f_{n+1}||_{-1}^{2}
\end{gather*}
\end{theorem}

\begin{proof}
We follow the stability analysis in the last proof. Take an inner product of
the first equation with $k_{n+1}u_{n+1}^{\varepsilon }$, the second with $%
k_{n+1}p_{n+1}^{\varepsilon }$, integrate over the flow domain, integrate by
parts, use skew symmetry, use the polarization identity twice and add. This
yields 
\begin{gather*}
\frac{1}{2}\int |u_{n+1}^{\varepsilon }|^{2}-|u_{n}^{\varepsilon
}|^{2}+|u_{n+1}^{\varepsilon }-u_{n}^{\varepsilon }|^{2}dx+\int k_{n+1}\nu
|\nabla u_{n+1}^{\varepsilon }|^{2}dx \\
\frac{1}{2}\int \varepsilon _{n+1}p_{n+1}^{\varepsilon 2}-2\varepsilon
_{n}p_{n}^{\varepsilon }p_{n+1}^{\varepsilon }+\varepsilon
_{n}p_{n+1}^{\varepsilon 2}dx=k_{n+1}\int u_{n+1}^{\varepsilon }\cdot
f_{n+1}dx.
\end{gather*}%
From the identity 
\begin{equation*}
2\varepsilon _{n}p_{n}^{\varepsilon }p_{n+1}^{\varepsilon }=\varepsilon
_{n}p_{n+1}^{\varepsilon 2}+\varepsilon _{n}p_{n}^{\varepsilon
2}-\varepsilon _{n}(p_{n+1}^{\varepsilon }-p_{n}^{\varepsilon })^{2},
\end{equation*}%
the energy equality becomes 
\begin{gather*}
\int \frac{1}{2}|u_{n+1}^{\varepsilon }|^{2}+\frac{\varepsilon _{n+1}}{2}%
|p_{n+1}^{\varepsilon }|^{2}dx-\int \frac{1}{2}|u_{n}^{\varepsilon }|^{2}+%
\frac{\varepsilon _{n}}{2}|p_{n}^{\varepsilon }|^{2}dx+\int k_{n+1}\nu
|\nabla u_{n+1}^{\varepsilon }|^{2}dx \\
\int \frac{1}{2}(u_{n+1}^{\varepsilon }-u_{n}^{\varepsilon })^{2}+\frac{%
\varepsilon _{n+1}}{2}(p_{n+1}^{\varepsilon }-p_{n}^{\varepsilon
})^{2}dx=k_{n+1}\int u_{n+1}^{\varepsilon }\cdot f_{n+1}dx.
\end{gather*}%
Upon summation the first two terms telescope, completing the proof of the
energy equality. The stability estimate follows from the energy equality and
the Cauchy-Schwarz-Young inequality.
\end{proof}

\subsection{Insight into a possible variable-$\protect\varepsilon $
instability}

The difficulty in ensuring long time stability when simply solving (\ref{AC}%
) for variable $\varepsilon$ can be understood at the level of the continuum
model. When $f=0$ the NSE kinetic energy is monotonically decreasing so any
growth in model energy represents an instability. Dropping the superscript $%
\varepsilon $\ for this sub-section, consider the kinetic energy evolution
of 
\begin{equation}
\begin{cases}
& \partial _{t}u+\nabla p=\nu \Delta u-(u\cdot \nabla )u-{\frac{1}{2}}%
(\nabla \cdot u)u, \\ 
& \varepsilon (t)\partial _{t}p+\nabla \cdot u=0,%
\end{cases}
\label{eq:VariableEpsilonModel}
\end{equation}%
subject to periodic or no slip boundary conditions. Computing the model's
kinetic energy by taking the inner product with, respectively, $u$ and $p$,
integrating then adding gives the continuum equivalent of the kinetic energy
law of the standard AC method (\ref{eq:KElawStandardAC}) above: 
\begin{equation}
\frac{d}{dt}\frac{1}{2}\int |u(t)|^{2}+\varepsilon (t)p(t)^{2}dx+\int \nu
|\nabla u|^{2}dx=\mathbf{+}\int \varepsilon ^{\prime }(t)p(t)^{2}dx.
\label{eq:VariableEnergyEq1}
\end{equation}%
The RHS suggests the following:

\textit{Decreasing }$\varepsilon $\textit{\ (}$\varepsilon ^{\prime }(t)<0$%
\textit{) acts to decrease the }$L^{2}$\textit{\ norm of }$u$\textit{\ and }$%
p$\textit{\ while increasing \ }$\varepsilon $\textit{\ (}$\varepsilon
^{\prime }(t)>0$\textit{) acts to increase the }$\mathit{L}^{2}$\textit{\
norm of }$u$\textit{\ and }$p$\textit{.}

Thus it seems like an example of instability would be simple to generate by
taking a solution with large pressure, small velocity, small $\nu $ and $%
\varepsilon ^{\prime }(t)>>\varepsilon (t)$. However, consider next the
equation for pressure fluctuations about a rest state. Beginning with 
\begin{equation}
\partial _{t}u+\nabla p=0{\text{ and }}\varepsilon (t)\partial _{t}p+\nabla
\cdot u=0,
\end{equation}%
eliminate the velocity in the standard manner for deriving the acoustic
equation. This yields the following induced equation for acoustic pressure
oscillations 
\begin{equation*}
\left( \varepsilon (t)p_{t}\right) _{t}-\triangle p=0.
\end{equation*}%
Oddly, $\varepsilon (t)=t$ (increasing) occurs in \cite{L84}. Multiplying by 
$p_{t}$ and integrating yields 
\begin{equation}
\frac{d}{dt}\int \varepsilon (t)(p_{t})^{2}+|\nabla p|^{2}dx=\mathbf{-}\int
\varepsilon _{t}(t)(p_{t})^{2}dx.  \label{eq:EnergyEq2}
\end{equation}%
The RHS of (\ref{eq:EnergyEq2}) yields the nearly opposite prediction that

\textit{Decreasing }$\varepsilon $\textit{\ (}$\varepsilon ^{\prime }(t)<0$%
\textit{) acts to increase the }$L^{2}$\textit{\ norm of }$p_{t}$\textit{\
and }$\nabla p$\textit{\ while increasing \ }$\varepsilon $\textit{\ (}$%
\varepsilon ^{\prime }(t)>0$\textit{) acts to increase the }$L^{2}$\textit{\
norm of }$p_{t}$\textit{\ and }$\nabla p$\textit{.}

The analytical conclusion is that long time stability of the standard AC
method with variable-$\varepsilon ,k$ is a murky open problem.

\section{Analysis of the variable-$\protect\varepsilon$ continuum AC Model}

The last subsection suggests that insight into the new model may be obtained
through analysis of its continuum analog without the assumption of small
fluctuations about a rest state. Accordingly, this section considers the
pure Cauchy problem, $\Omega ={{\mathbb{R}}}^{3}$, for%
\begin{equation}
\begin{cases}
& \partial _{t}u^{\varepsilon }+\nabla p^{\varepsilon }=\nu \Delta
u^{\varepsilon }-(u^{\varepsilon }\cdot \nabla )u^{\varepsilon }-{\frac{1}{2}%
}(\nabla \cdot u^{\varepsilon })u^{\varepsilon }+f^{\varepsilon } \\ 
& \partial _{t}(\varepsilon (t)p^{\varepsilon })-{\frac{1}{2}}\varepsilon
_{t}(t)p^{\varepsilon }+\nabla \cdot u^{\varepsilon }=0.%
\end{cases}%
\end{equation}%
To explain the change of the pressure term in the continuity equation from $%
\varepsilon p_{t}^{\varepsilon }$ to $\partial _{t}(\varepsilon
(t)p^{\varepsilon })-{\frac{1}{2}}\varepsilon _{t}(t)p^{\varepsilon }$, note
that 
\begin{equation*}
\left( {\frac{1}{2}}\varepsilon (p^{\varepsilon })^{2}\right)
_{t}=p^{\varepsilon }\left[ (\varepsilon p^{\varepsilon })_{t}-\frac{1}{2}%
\varepsilon _{t}p^{\varepsilon }\right] =p^{\varepsilon }\left[ \varepsilon
p_{t}^{\varepsilon }+{\frac{1}{2}}\varepsilon _{t}p^{\varepsilon }\right] .
\end{equation*}%
This can equivalently be formulated as $\frac{1}{2}(\varepsilon
p^{\varepsilon })_{t}+\frac{1}{2}\varepsilon p_{t}^{\varepsilon }$\ since 
\begin{equation*}
\left( {\frac{1}{2}}\varepsilon (p^{\varepsilon })^{2}\right)
_{t}=p^{\varepsilon }\frac{1}{2}\left[ (\varepsilon p^{\varepsilon
})_{t}+\varepsilon p_{t}^{\varepsilon }\right] .
\end{equation*}

We will first recall the notion of Leray weak solution of the NS equation,
and then derive the basic energy estimate for the new AC system (1.3), which
will lead to the appropriate assumptions on initial conditions. We then
introduce the assumption on the variable $\varepsilon (t)$, based on which
we perform a dispersive approach to obtain the Strichartz estimate for the
pressure.

\subsection{Leray weak solution for NSE}

We analyze the $\varepsilon \rightarrow 0$ limit of the continuum AC model %
\eqref {ACm}. Since we will be focused on the convergence of the
approximated system to a weak solution of the NSE, from now on we will for
simplicity take $\nu =1$ and $f=0$. The inclusion of a body force and a
different value of the kinematic viscosity adds no technical difficulty to
the analysis.

Let us recall the notion of a Leray weak solution (see, for e.g. Lions \cite%
{lions1996mathematical} and Temam \cite{temam2001book}) of the NSE.

\begin{definition}
We say that $u\in L^{\infty }([0,T];L^{2}({{\mathbb{R}}}^{3}))\cap
L^{2}([0,T];\dot{H}^{1}({{\mathbb{R}}}^{3}))$ is a Leray weak solution of
the NS equation if it satisfies \eqref {NS} in the sense of distribution for
all test functions $\varphi \in C_{0}^{\infty }([0,T]\times {{\mathbb{R}}}%
^{3})$ with $\nabla \cdot \varphi =0$ and moreover the following energy
inequality holds for every $t\in \lbrack 0,T]$ 
\begin{equation}
\begin{split}
{\frac{1}{2}}& \int_{{{\mathbb{R}}}^{3}}|u(t,x)|^{2}\,dx+\nu
\int_{0}^{t}\int_{{{\mathbb{R}}}^{3}}|\nabla u(s,x)|^{2}\,dxds \\
& \leq {\frac{1}{2}}\int_{{{\mathbb{R}}}^{3}}|u(0,x)|^{2}\,dx.
\end{split}
\label{energyinequ}
\end{equation}
\end{definition}

\subsection{Energy estimates}

We can easily verify that system \eqref {ACm} obeys the classical energy
type estimate.

\begin{theorem}
\label{thm_energy_m} Let $(u^{\varepsilon },p^{\varepsilon })$ be a strong solution
to \eqref {ACm} on $[0,T]$. Then it follows that for all $t\in [0, T]$
\begin{equation}  \label{energyestm}
E(t)+\int _0^t\int _{{{\mathbb{R}}}^3}|\nabla u^{\varepsilon
}(s,x)|^2\,dxds=E(0),
\end{equation}
where 
\begin{equation}  \label{energyfunc}
E(t)={\frac{1}{2}}\int _{{{\mathbb{R}}}^3}\left (|u^{\varepsilon
}(t,x)|^2+\varepsilon (t)|p^{\varepsilon }(t,x)|^2\right )\,dx.
\end{equation}
\end{theorem}

Since we expect the approximated solution $(u^{\varepsilon}, p^{\varepsilon})$ to converge to the Leray solution, we require the finite energy constraint to be satisfied by $(u^{\varepsilon}, p^{\varepsilon})$. So following \cite{donatelli2006whole} we further restrict the initial condition to system \eqref {ACm}
(or \eqref {AC}) to satisfy 
\begin{equation}  \label{initialcond}
\begin{cases}
& u_0^{\varepsilon }:=u^{\varepsilon }(0,\cdot )\to u_0{{\text { strongly in 
}}}L^2({{\mathbb{R}}}^3){{\text { as }}}\varepsilon \to 0, \\ 
& \sqrt{\varepsilon }p_0^{\varepsilon }:=\sqrt{\varepsilon }p^{\varepsilon
}(0,\cdot )\to 0{{\text { strongly in }}}L^2({{\mathbb{R}}}^3){{\text { as }}%
}\varepsilon \to 0.%
\end{cases}%
\end{equation}

This way we can obtain the following uniform estimates which are similar to
those in \cite[Corollary 4.2]{donatelli2006whole}, and hence we omit the
proof.
\begin{corollary}
\label{cor_unif est} Under the assumptions of Theorem \ref{thm_energy_m},
together with \eqref {initialcond}, it follows that 
\begin{align}
&\sqrt{\varepsilon }p^{\varepsilon }\quad & {{\text {is bounded in }}}%
L^{\infty }([0,T];L^2({{\mathbb{R}}}^3)),  \label{pressure} \\
&\varepsilon p_t^{\varepsilon } & {{\text {is relatively compact in }}}%
H^{-1}([0,T]\times {{\mathbb{R}}}^3),  \label{pressure_t} \\
&\nabla u^{\varepsilon } & {{\text {is bounded in }}}L^2([0,T]\times {{%
\mathbb{R}}}^3),  \label{gradu} \\
&u^{\varepsilon } & {{\text { is bounded in }}}L^{\infty }([0,T];L^2({{%
\mathbb{R}}}^3))\cap L^2([0,T];L^6({{\mathbb{R}}}^3)),  \label{u} \\
&(u^{\varepsilon }\cdot \nabla )u^{\varepsilon } & {{\text {is bounded in }}}%
L^2([0,T];L^1({{\mathbb{R}}}^3))\cap L^1([0,T];L^{3/2}({{\mathbb{R}}}^3)),
\label{convective} \\
&(\nabla \cdot u^{\varepsilon })u^{\varepsilon } & {{\text {is bounded in }}}%
L^2([0,T];L^1({{\mathbb{R}}}^3))\cap L^1([0,T];L^{3/2}({{\mathbb{R}}}^3)).
\label{correction}
\end{align}
\end{corollary}

\subsection{Acoustic pressure wave and Strichartz estimates}

\label{subsec_pressure}

Note that we can derive from system \eqref {ACm} that the pressure $%
p^{\varepsilon }$ satisfies the following wave equations 
\begin{equation}  \label{PWm}
(\varepsilon p^{\varepsilon })_{tt}-\left ({\frac{1}{2}}\varepsilon
_tp^{\varepsilon }\right )_t-\Delta p^{\varepsilon }=-\Delta (\nabla \cdot
u^{\varepsilon })+\nabla \cdot \left [(u^{\varepsilon }\cdot \nabla
)u^{\varepsilon }+{\frac{1}{2}}(\nabla \cdot u^{\varepsilon })u^{\varepsilon
}\right ].
\end{equation}
Our assumption on the relaxation parameter $\varepsilon (t)$ is 
\begin{equation}
\varepsilon (t)\in C^2([0,T]),\quad 0<c\epsilon \le \varepsilon (t)\leq
C\epsilon ,\quad \varepsilon _{t}(t),\varepsilon _{tt}(t)\sim o(\varepsilon
(t)),  \label{assumponepsilon}
\end{equation}
for $t\in [0,T]$, and $\epsilon $, $c$ and $C$ are some positive constants.
The last condition in the above is understood as 
\begin{equation*}
\lim _{\epsilon \to 0}{\frac{\varepsilon _t(t)}{\varepsilon (t)}}=0,\quad
\lim _{\epsilon \to 0}{\frac{\varepsilon _{tt}(t)}{\varepsilon (t)}}=0.
\end{equation*}

From the assumption \eqref {assumponepsilon} we may write 
\begin{equation}  \label{rewriteepsilon}
\varepsilon (t)=\epsilon A(t)
\end{equation}
for some function $A(t)$ satisfying 
\begin{equation}  \label{assumponA}
A\in C^2([0,T]),\quad c\le A(t)\le C,\quad A^{\prime }(t),A^{\prime \prime
}(t)\sim o(A(t)).
\end{equation}
Performing the following rescaling 
\begin{equation}  \label{rescale}
\tau ={\frac{t}{\sqrt{\epsilon }}},\ \ \tilde {p}(\tau ,x)=p^{\varepsilon }(%
\sqrt{\epsilon }\tau ,x),\ \ \tilde {u}(\tau ,x)=u^{\varepsilon }(\sqrt{%
\epsilon }\tau ,x),\ \ \tilde {A}(\tau )=A(\sqrt{\epsilon }\tau ),
\end{equation}
and plugging into \eqref {PWm} we obtain 
\begin{equation}  \label{PWnewm}
(\tilde {A}\tilde {p})_{\tau \tau }-\Delta \tilde {p}-\left ({\frac{1}{2}}%
\tilde {A}_{\tau }\tilde {p}\right )_{\tau }=-\Delta (\nabla \cdot \tilde {u}%
)+\nabla \cdot \left [(\tilde {u}\cdot \nabla )\tilde {u}+{\frac{1}{2}}%
(\nabla \cdot \tilde {u})\tilde {u}\right ].
\end{equation}

Note that here the wave operator contains time-dependent coefficients. The
space-time Strichartz estimates involving variable coefficients were
established by Mockenhaupt et al \cite{Mockenhaupt1993} when the
coefficients are smooth. Operators with $C^{1,1}$ coefficients were first
considered by Smith \cite{Smith1998} using wave packets. An alternative
method based on the FBI transform was later employed by Tataru \cite%
{tataru2000,tataru2001,tataru2002} to prove the full range of Strichartz
estimates under weaker assumptions. Here we briefly state the estimates we
need for \eqref {PWnewm}.

\begin{theorem}
\textbf{The results of Tataru \cite{tataru2002}}. \label{thm_strichartz}
Let $w$ be a (weak) solution of the following wave equations in $[0,T]\times 
{{\mathbb{R}}}^{n}$ 
\begin{equation}
\begin{cases}
& (Aw)_{tt}-\left( {\frac{1}{2}}A_{t}w\right) _{t}-\Delta w=F(t,x), \\ 
& w(0,\cdot )=w_{0},\quad w_{t}(0,\cdot )=w_{1}.%
\end{cases}
\label{wavem}
\end{equation}%
Assume that $A(t)>0$ and $\nabla _{t,x}^{2}A\in L^{1}([0,T];L^{\infty }({{%
\mathbb{R}}}^{n}))$. Then the following Strichartz estimates hold 
\begin{equation}
\Vert w\Vert _{L_{t}^{q}L_{x}^{r}}+\Vert w_{t}\Vert
_{L_{t}^{q}W_{x}^{-1,r}}\lesssim \Vert w_{0}\Vert _{\dot{H}_{x}^{\gamma
}}+\Vert w_{1}\Vert _{\dot{H}_{x}^{\gamma -1}}+\Vert F\Vert _{L_{t}^{\tilde{q%
}^{\prime }}L_{x}^{\tilde{r}^{\prime }}},  \label{strichartz_gen}
\end{equation}%
where $(q,r,\gamma )$ and $(\tilde{q}^{\prime },\tilde{r}^{\prime })$
satisfy 
\begin{equation}
\begin{cases}
& 2\leq q,r\leq \infty , \\ 
&  \\ 
& (q,r,\gamma ),(\tilde{q}^{\prime },\tilde{r}^{\prime },\gamma )\neq
(2,\infty ,1),{{\text{ when }}}n=3, \\ 
&  \\ 
& \displaystyle{\frac{1}{q}}+{\frac{n}{r}}={\frac{n}{2}}-\gamma ={\frac{1}{%
\tilde{q}^{\prime }}}+{\frac{n}{\tilde{r}^{\prime }}}-2, \\ 
&  \\ 
& \displaystyle{\frac{2}{q}}+{\frac{n-1}{r}}\leq {\frac{n-1}{2}},\quad {%
\frac{2}{\tilde{q}}}+{\frac{n-1}{\tilde{r}}}\leq {\frac{n-1}{2}}.%
\end{cases}
\label{admissible}
\end{equation}
\end{theorem}

For our purpose, $n=3$, and we will take $(q,r)=(4,4)$, $(\tilde{q}^{\prime
},\tilde{r}^{\prime })=(1,3/2)$, and $\gamma =1/2$. This way the above
Strichartz estimate becomes 
\begin{equation}
\Vert w||_{L_{t,x}^{4}}+\Vert w_{t}\Vert _{L_{t}^{4}W_{x}^{-1,4}}\lesssim
\Vert w_{0}\Vert _{\dot{H}_{x}^{\frac{1}{2}}}+\Vert w_{1}\Vert _{\dot{H}%
_{x}^{-{\frac{1}{2}}}}+\Vert F\Vert _{L_{t}^{1}L_{x}^{\frac{3}{2}}}.
\label{strichartz}
\end{equation}

Following \cite{donatelli2006whole}, we decompose the pressure as $\tilde {p}%
=\tilde {p}_1+\tilde {p}_2$ where 
\begin{equation}  \label{eqnp_1}
\begin{cases}
& (\tilde {A}\tilde {p}_1)_{\tau \tau }-\left ({\frac{1}{2}}\tilde {A}_{\tau
}\tilde {p}_1\right )_{\tau }-\Delta \tilde {p}_1=\nabla \cdot \left [(%
\tilde {u}\cdot \nabla )\tilde {u}+{\frac{1}{2}}(\nabla \cdot \tilde {u})%
\tilde {u}\right ]=:\nabla \cdot \tilde {F}, \\ 
& \tilde {p}_1(x,0)=\tilde {p}(x,0),\quad \partial _{\tau }\tilde {p}%
_1(x,0)=\partial _{\tau }\tilde {p}(x,0),%
\end{cases}%
\end{equation}

\begin{equation}
\begin{cases}
& (\tilde{A}\tilde{p}_{2})_{\tau \tau }-\left( {\frac{1}{2}}\tilde{A}_{\tau }%
\tilde{p}_{2}\right) _{\tau }-\Delta \tilde{p}_{2}=-\Delta (\nabla \cdot 
\tilde{u}), \\ 
& \tilde{p}_{2}(x,0)=\partial _{\tau }\tilde{p}_{2}(x,0)=0.%
\end{cases}
\label{eqnp_2}
\end{equation}%
Applying Theorem \ref{thm_strichartz} to the above two systems and
unraveling the change-of-\newline
variables \eqref {rescale} we obtain the following estimates.

\begin{theorem}
\label{thm_est pressure} Let $(u^{\varepsilon },p^{\varepsilon })$ be a strong
solution of the Cauchy problem on $[0,T]$ to system \eqref {ACm} with initial data $%
(u_{0}^{\varepsilon },p_{0}^{\varepsilon })$ satisfying \eqref {initialcond}%
. Assume also that $\varepsilon (t)$ satisfies (\ref{condonepsilon}). Then
for $\epsilon $ small enough the following estimate holds. 
\begin{equation}
\begin{split}
\epsilon ^{{\frac{3}{8}}}\Vert p^{\varepsilon }\Vert _{L_{t}^{4}W_{x}^{-{2}%
,4}}+\epsilon ^{{\frac{7}{8}}}\Vert p_{t}^{\varepsilon }\Vert
_{L_{t}^{4}W_{x}^{-{3},4}}\lesssim & \ \sqrt{\epsilon }\Vert
p_{0}^{\varepsilon }\Vert _{L_{x}^{2}}+\Vert \nabla \cdot u_{0}^{\varepsilon
}\Vert _{H_{x}^{-1}}+\sqrt{T}\Vert \nabla \cdot u^{\varepsilon }\Vert
_{L_{t,x}^{2}} \\
& +\left\Vert (u^{\varepsilon }\cdot \nabla )u^{\varepsilon }+{\frac{1}{2}}%
(\nabla \cdot u^{\varepsilon })u^{\varepsilon }\right\Vert _{L_{t}^{1}L_{x}^{%
{\frac{3}{2}}}}.
\end{split}
\label{strichartzforp}
\end{equation}
\end{theorem}

\begin{proof}
We first apply \eqref {strichartz} with $w=\Delta ^{-1/2}\tilde{p}_{2}$ to
obtain 
\begin{equation}
\Vert \tilde{p}_{1}\Vert _{L_{\tau }^{4}W_{x}^{-1,4}}+\Vert \partial _{\tau }%
\tilde{p}_{1}\Vert _{L_{\tau }^{4}W_{x}^{-2,4}}\lesssim \Vert \tilde{p}%
(x,0)\Vert _{\dot{H}_{x}^{-{\frac{1}{2}}}}+\Vert \partial _{\tau }\tilde{p}%
(x,0)\Vert _{\dot{H}_{x}^{-{\frac{3}{2}}}}+\Vert \tilde{F}\Vert _{L_{\tau
}^{1}L_{x}^{{\frac{3}{2}}}}.  \label{tildep_1}
\end{equation}%
The estimate for $\tilde{p}_{2}$ requires some more effort since from the
energy we only have $L^{2}$-control of $\nabla \cdot u^{\varepsilon }$. For
that, we further introduce a time-scale change 
\begin{equation*}
\tau =\beta (s),\quad \bar{p}(s,x)=\tilde{p}_{2}(\beta (s),x),\quad \bar{u}%
(s,x)=\tilde{u}(\beta (s),x),\quad a(s)=\tilde{A}(\beta (s)).
\end{equation*}%
By choosing $\beta $ such that 
\begin{equation}
\beta ^{\prime }(s)=\sqrt{a(s)}>0,\quad \beta (0)=0,  \label{beta}
\end{equation}%
we can rewrite \eqref {eqnp_2} as 
\begin{equation}
\begin{cases}
& \bar{p}_{ss}-\Delta \bar{p}=-\Delta (\nabla \cdot \bar{u})-M\bar{p}_{s}-N%
\bar{p}, \\ 
& \bar{p}(x,0)=\bar{p}_{s}(x,0)=0,%
\end{cases}
\label{eqnnewp_2}
\end{equation}%
where 
\begin{equation}
M(s):={\frac{3a^{\prime }}{2a}}-{\frac{a^{\prime }}{2\sqrt{a}}},\quad N(s):={%
\frac{1}{2}}\left( {\frac{a^{\prime \prime }}{a}}-{\frac{(a^{\prime 2}}{%
2a^{3/2}}}\right) .  \label{MN}
\end{equation}%
The advantage of this transformation is that one can treat the transformed
problem as the classical inhomogeneous wave equation where $\bar{p}$ and $%
\bar{p}_{s}$ on the right-hand side are considered as forcing terms. A
direction computation yields that 
\begin{equation*}
M(s)={\frac{3-\sqrt{A(t)}}{2\sqrt{A(t)}}}A^{\prime }(t)\sqrt{\epsilon }%
,\qquad N(s)={\frac{\left[ 2A^{\prime \prime }(t)\sqrt{A(t)}+(1-\sqrt{A(t)}%
)(A^{\prime 2}\right] }{4A(t)}}\epsilon .
\end{equation*}%
Here $t=\sqrt{\epsilon }\beta (s)$. 

Applying the Strichartz estimates for the classical wave operator in three
spatial dimensions we have (see, for instance \cite{S95}) 
\begin{align}
\Vert \bar{p}\Vert _{L_{s,x}^{4}}+\Vert \bar{p}_{s}\Vert
_{L_{s}^{4}W_{x}^{-1,4}}& \leq C_{1}\left( \Vert \Delta (\nabla \cdot \bar{u}%
)+M\bar{p}_{s}+N\bar{p}\Vert _{L_{s}^{1}L_{x}^{2}}\right) ,  \label{Sogge1}
\\
\Vert \bar{p}\Vert _{L_{s}^{\infty }L_{x}^{6}}+\Vert \bar{p}\Vert
_{L_{x}^{\infty }\dot{H}_{x}^{1}}+\Vert \bar{p}_{s}\Vert _{L_{s}^{\infty
}L_{x}^{2}}& \leq C_{2}\left( \Vert \Delta (\nabla \cdot \bar{u})+M\bar{p}%
_{s}+N\bar{p}\Vert _{L_{s}^{1}L_{x}^{2}}\right) ,  \label{Sogge2}
\end{align}%
where $C_{1}$ and $C_{2}$ are universal constants. It's easily seen that the
following bound holds 
\begin{equation}
\Vert M\bar{p}_{s}\Vert _{L_{s}^{1}L_{x}^{2}}\leq \Vert M\Vert
_{L_{s}^{1}}\Vert \bar{p}_{s}\Vert _{L_{s}^{\infty }L_{x}^{2}}=\Vert \bar{p}%
_{s}\Vert _{L_{s}^{\infty }L_{x}^{2}}\int_{0}^{T}\left\vert {\frac{3-\sqrt{A}%
}{2\sqrt{A}}}A^{\prime }\right\vert (t)\,dt.  \label{estM}
\end{equation}%
To bound $N\bar{p}$, we use that $\displaystyle\bar{p}(s)=\int_{0}^{s}\bar{p}%
_{s}\,ds$, and so 
\begin{equation*}
|\bar{p}(s)|\leq \sqrt{s}\left( \int_{0}^{s}|\bar{p}_{s}|^{2}\,ds\right)
^{1/2}.
\end{equation*}%
Thus from the above and the Fubini theorem it follows that 
\begin{align*}
\Vert \bar{p}_{s}(s,\cdot )\Vert _{L_{x}^{2}}^{2}& \leq s\int_{{\mathbb{R}}%
^{3}}\int_{0}^{s}|\bar{p}_{s}|^{2}\,dsdx=s\int_{0}^{s}\int_{{\mathbb{R}}%
^{3}}|\bar{p}_{s}|^{2}\,dsdx \\
& \leq s\int_{0}^{s}\Vert \bar{p}_{s}\Vert _{L_{s}^{\infty
}L_{x}^{2}}^{2}\,ds=s^{2}\Vert \bar{p}_{s}\Vert _{L_{s}^{\infty
}L_{x}^{2}}^{2}.
\end{align*}%
Therefore 
\begin{equation}
\begin{split}
\Vert N\bar{p}\Vert _{L_{s}^{1}L_{x}^{2}}& \leq \Vert \bar{p}_{s}\Vert
_{L_{s}^{\infty }L_{x}^{2}}\int_{0}^{\beta ^{-1}(T/\sqrt{\epsilon }%
)}s|N(s)|\,ds \\
& \leq {\frac{\sqrt{\epsilon }}{2}}\beta ^{-1}\left( {\frac{T}{\sqrt{%
\epsilon }}}\right) \Vert \bar{p}_{s}\Vert _{L_{s}^{\infty
}L_{x}^{2}}\int_{0}^{T}\left\vert {\frac{\left[ 2A^{\prime \prime }\sqrt{A}%
+(1-\sqrt{A})(A^{\prime })^{2}\right] }{4A}}\right\vert (t)\,dt \\
& \leq {\frac{T}{2\sqrt{c}}}\Vert \bar{p}_{s}\Vert _{L_{s}^{\infty
}L_{x}^{2}}\int_{0}^{T}\left\vert {\frac{\left[ 2A^{\prime \prime }\sqrt{A}%
+(1-\sqrt{A})(A^{\prime })^{2}\right] }{4A}}\right\vert (t)\,dt,
\end{split}
\label{estN}
\end{equation}%
where the last inequality is due to the following estimate 
\begin{equation*}
(\beta ^{-1})^{\prime }={\frac{1}{\sqrt{a}}}\leq {\frac{1}{\sqrt{c}}}.
\end{equation*}%
Recall from (\ref{assumponA}) that $A^{\prime },A^{\prime \prime }\sim
o(A(t))$. Therefore for sufficiently small $\varepsilon $, from (\ref{estM})
and (\ref{estN}) we obtain that 
\begin{equation*}
\Vert M\bar{p}\Vert _{L_{s}^{1}L_{x}^{2}}+\Vert N\bar{p}\Vert
_{L_{s}^{1}L_{x}^{2}}\leq {\frac{1}{2C_{2}}}\Vert \bar{p}_{s}\Vert
_{L_{s}^{\infty }L_{x}^{2}}.
\end{equation*}%
Hence from (\ref{Sogge2}) it further yields that 
\begin{equation*}
\Vert \bar{p}_{s}\Vert _{L_{s}^{\infty }L_{x}^{2}}\leq 2C_{2}\Vert \Delta
(\nabla \cdot \bar{u})\Vert _{L_{s}^{1}L_{x}^{2}}.
\end{equation*}%
Thus from (\ref{Sogge1}) we finally have 
\begin{equation*}
\Vert \bar{p}\Vert _{L_{s,x}^{4}}+\Vert \bar{p}_{s}\Vert
_{L_{s}^{4}W_{x}^{-1,4}}\leq 2C_{1}\Vert \Delta (\nabla \cdot \bar{u})\Vert
_{L_{s}^{1}L_{x}^{2}}\leq {\frac{2C_{1}\sqrt{T}}{(c\epsilon )^{1/4}}}\Vert
\Delta (\nabla \cdot \bar{u})\Vert _{L_{s}^{2}L_{x}^{2}}.
\end{equation*}%
A similar argument applied to $\Delta ^{-1}\bar{p}$ implies that 
\begin{equation}
\Vert \bar{p}\Vert _{L_{s}^{4}W_{x}^{-2,4}}+\Vert \bar{p}_{s}\Vert
_{L_{s}^{4}W_{x}^{-3,4}}\leq {\frac{2C_{1}\sqrt{T}}{(c\epsilon )^{1/4}}}%
\Vert \nabla \cdot \bar{u}\Vert _{L_{s}^{2}L_{x}^{2}}.  \label{estbarp}
\end{equation}%
Note that 
\begin{equation*}
\Vert \bar{p}\Vert _{L_{s}^{r}}=\Vert \tilde{A}^{-{1/2r}}\tilde{p}_{2}\Vert
_{L_{\tau }^{r}},\quad \bar{p}_{s}=\sqrt{\tilde{A}}\partial _{\tau }\tilde{p}%
_{2}.
\end{equation*}%
Hence from (\ref{estbarp}) we have 
\begin{equation}
\Vert \tilde{p}_{2}\Vert _{L_{\tau }^{4}W_{x}^{-2,4}}+\Vert \partial _{\tau
}p_{2}\Vert _{L_{\tau }^{4}W_{x}^{-3,4}}\lesssim {\frac{\sqrt{T}}{\epsilon
^{1/4}}}\Vert \nabla \cdot \tilde{u}\Vert _{L_{\tau }^{2}L_{x}^{2}}.
\label{tildep_2}
\end{equation}%
Putting together (\ref{tildep_1}) and (\ref{tildep_2}) we have that 
\begin{equation*}
\begin{split}
\Vert \tilde{p}\Vert _{L_{\tau }^{4}W_{x}^{-2,4}}+\Vert \partial _{\tau }%
\tilde{p}\Vert _{L_{\tau }^{4}W_{x}^{-{3},4}}\lesssim & \ \Vert \tilde{p}%
(x,0)\Vert _{\dot{H}_{x}^{-{\frac{1}{2}}}}+\Vert \partial _{\tau }\tilde{p}%
(x,0)\Vert _{\dot{H}_{x}^{-{\frac{3}{2}}}} \\
& +{\frac{\sqrt{T}}{{\epsilon }^{1/4}}}\Vert \nabla \cdot \tilde{u}\Vert
_{L_{\tau }^{2}L_{x}^{2}}+\Vert \tilde{F}\Vert _{L_{\tau }^{1}L_{x}^{{\frac{3%
}{2}}}}.
\end{split}%
\end{equation*}%
Notice that first we have from the rescaling (\ref{rescale}) that 
\begin{equation*}
\tilde{f}_{\tau }=\sqrt{\epsilon }f_{t}^{\varepsilon },\quad \Vert \tilde{f}%
\Vert _{L_{\tau }^{r}L_{x}^{q}}=\epsilon ^{-1/2r}\Vert f^{\varepsilon }\Vert
_{L_{t}^{r}L_{x}^{q}}.
\end{equation*}%
Also, from the second equation in (\ref{ACm}) and (\ref{rewriteepsilon}) we
find that 
\begin{equation}
p_{t}^{\varepsilon }=-\left( {\frac{\nabla \cdot u^{\varepsilon }}{\epsilon
A(t)}}+{\frac{A^{\prime }(t)}{2A(t)}}p^{\varepsilon }\right) .
\label{p_texp}
\end{equation}%
Therefore we can estimate 
\begin{equation*}
\begin{split}
\Vert \partial _{\tau }\tilde{p}(x,0)\Vert _{\dot{H}_{x}^{-{\frac{3}{2}}}}&
\leq \Vert \partial _{\tau }\tilde{p}(x,0)\Vert _{\dot{H}_{x}^{-{1}}}=\sqrt{%
\epsilon }\Vert p_{t}^{\varepsilon }(x,0)\Vert _{\dot{H}_{x}^{-{1}}} \\
& \lesssim \epsilon ^{-1/2}\Vert \nabla \cdot u_{0}^{\varepsilon }\Vert
_{H_{x}^{-1}}+\Vert p_{0}^{\varepsilon }\Vert _{L_{x}^{2}}.
\end{split}%
\end{equation*}%
Putting all the above together we derive (\ref{strichartzforp}).
\end{proof}

Given the a priori energy estimates Theorem \ref{thm_energy_m} and the pressure estimates Theorem \ref{thm_est pressure}, we can now use the Galerkin approximation method to obtain the global existence of weak solutions to system \eqref{ACm}. Since the proof is quite standard we will omit it here.
\begin{theorem}\label{thm_weak}
Let $\varepsilon(t) > 0$ and $(u_0^{\varepsilon }, p_0^{\varepsilon })$ satisfy condition \eqref{initialcond}. Then for any $T>0$, system \eqref{ACm} admits a weak solution $(u^{\varepsilon}, p^{\varepsilon})$ with the following properties
\begin{enumerate}
\item[(1)] $u^{\varepsilon} \in L^{\infty }([0,T];L^{2}({{\mathbb{R}}}^{3}))\cap L^{2}([0,T];\dot{H}^{1}({{\mathbb{R}}}^{3}))$;
\item[(2)] $\sqrt{\varepsilon} p^{\varepsilon} \in L^\infty([0,T]; L^2(\mathbb{R}^3))$.
\end{enumerate}
\end{theorem}

\section{Convergence to the NSE}

The goal of this section is to establish the convergence of the AC system (%
\ref{ACm}) to the NS system, cf. Theorem \ref{thm_main conv}. The key step
is to show the strong convergence of the gradient part and the
divergence-free part of the velocity field. For this, let us denote ${%
\mathbb{P}}$ the Leray projection defined by 
\begin{equation}
{\mathbb{P}}=I-{\mathbb{Q}},\qquad {\text{where }}\qquad {\mathbb{Q}}=\nabla
(\Delta ^{-1}\nabla \cdot ).  \label{LerayP}
\end{equation}%
Note that ${\mathbb{P}}$ and ${\mathbb{Q}}$ are both bounded linear
operators on $W^{k,q}({\mathbb{R}}^{3})$ for any $k$ and $q\in (
1,\infty )$. See, e.g., \cite{stein2016harmonic}.

From Corollary \ref{cor_unif est} and Theorem \ref{thm_est pressure} we
easily obtain the following result.

\begin{proposition}
\label{prop_conv} Let the assumptions in Theorem \ref{thm_est pressure}
hold. Then as $\varepsilon \to 0$ it follows that 
\begin{align}
&\varepsilon p^{\varepsilon }\to 0\qquad {\text { strongly in }}L^{\infty
}([0,T];L^2({\mathbb{R}}^3))\cap L^4([0,T];W^{-{2},4}({\mathbb{R}}^3)),
\label{strongconvp} \\
&\nabla \cdot u^{\varepsilon }\to 0\quad {\text { strongly in }}W^{-1,\infty
}([0,T];L^2({\mathbb{R}}^3))\cap L^4([0,T];W^{-{3},4}({\mathbb{R}}^3)).
\label{strongconvdivu}
\end{align}
\end{proposition}

\begin{proof}
It is easily seen that (\ref{strongconvp}) follows from (\ref{pressure}), (%
\ref{strichartzforp}). Further, (\ref{strongconvdivu}) follows from (\ref%
{strichartzforp}) and the second equation of (\ref{ACm}).
\end{proof}

\subsection{Strong convergence of ${\mathbb{Q}}u^{\protect\varepsilon }$}

We will first prove that ${\mathbb{Q}}u^{\varepsilon }$ goes to zero in some
strong sense as $\varepsilon \to 0$.

\begin{lemma}
\label{lem_conv Q} Let $(u^{\varepsilon },p^{\varepsilon })$ be the solution
of the Cauchy problem to system (\ref{ACm}) with initial data $%
(u_{0}^{\varepsilon },p_{0}^{\varepsilon })$ satisfying (\ref{initialcond}).
Assume also that $\varepsilon (t)$ satisfies (\ref{assumponepsilon}). Then
for any $4\leq p<6$, 
\begin{equation}
{\mathbb{Q}}u^{\varepsilon }\rightarrow 0\quad {\text{ in }}\quad
L^{2}([0,T];L^{p}),\quad {\text{as }}\ \varepsilon \rightarrow 0.
\label{convQ}
\end{equation}
\end{lemma}

\begin{proof}
We follow the idea from \cite[Proposition 5.3]{donatelli2006whole}. Consider
the standard mollifier 
\begin{equation*}
\eta \in C_{0}^{\infty }({\mathbb{R}}^{3}),\ \eta \geq 0,\ \int_{{\mathbb{R}}%
^{3}}\eta \,dx=1;\qquad \eta _{\alpha }(x):=\alpha ^{-3}\eta (x/\alpha ),\
0<\alpha <1.
\end{equation*}%
Set $f_{\alpha }:=f\ast \eta _{\alpha }$. Then for any $f\in \dot{H}^{1}({%
\mathbb{R}}^{3})$ it holds 
\begin{equation}
\Vert f-f_{\alpha }\Vert _{L^{p}}\leq C\alpha ^{1-3\left( {\frac{1}{2}}-{%
\frac{1}{p}}\right) }\Vert \nabla f\Vert _{L^{2}},\quad \Vert f_{\alpha
}\Vert _{L^{r}}\leq C\alpha ^{-s-3\left( {\frac{1}{q}}-{\frac{1}{r}}\right)
}\Vert f\Vert _{W^{-s,q}}  \label{molest}
\end{equation}%
where $p\in \lbrack 2,6]$, $1\leq q\leq r\leq \infty $, $s\geq 0$. 

With the above, we decompose ${\mathbb{Q}}u^{\varepsilon }$ as 
\begin{equation*}
\Vert {\mathbb{Q}}u^{\varepsilon }\Vert _{L_{t}^{2}L_{x}^{p}}\leq \Vert {%
\mathbb{Q}}u^{\varepsilon }-({\mathbb{Q}}u^{\varepsilon })_{\alpha }\Vert
_{L_{t}^{2}L_{x}^{p}}+\Vert ({\mathbb{Q}}u^{\varepsilon })_{\alpha }\Vert
_{L_{t}^{2}L_{x}^{p}}=:J_{1}+J_{2}.
\end{equation*}%
Applying (\ref{molest}) to $J_{1}$ we have 
\begin{equation*}
J_{1}\leq C\alpha ^{1-3\left( {\frac{1}{2}}-{\frac{1}{p}}\right) }\left(
\int_{0}^{T}\Vert \nabla {\mathbb{Q}}u^{\varepsilon }\Vert
_{L_{x}^{2}}^{2}\,dt\right) ^{1/2}\leq C\alpha ^{1-3\left( {\frac{1}{2}}-{%
\frac{1}{p}}\right) }\Vert \nabla u^{\varepsilon }\Vert
_{L_{t}^{2}L_{x}^{2}}.
\end{equation*}%
As for $J_{2}$, from (\ref{p_texp}) we see that 
\begin{equation*}
{\mathbb{Q}}u^{\varepsilon }=\nabla \Delta ^{-1}(\nabla \cdot u^{\varepsilon
})=-\epsilon \nabla \Delta ^{-1}\left( Ap_{t}^{\varepsilon }+{\frac{1}{2}}%
A^{\prime }p^{\varepsilon }\right) .
\end{equation*}%
Thus from (\ref{molest}) we have 
\begin{align*}
J_{2}& =\epsilon \left\Vert \nabla \Delta ^{-1}\left( Ap_{t}^{\varepsilon }+{%
\frac{1}{2}}A^{\prime }p^{\varepsilon }\right) \ast \psi _{\alpha
}\right\Vert _{L_{t}^{2}L_{x}^{p}} \\
& \lesssim \epsilon \alpha ^{-{\frac{3}{2}}-3\left( {\frac{1}{4}}-{\frac{1}{p%
}}\right) }\Vert Ap_{t}^{\varepsilon }\Vert _{L_{t}^{2}W_{x}^{-{3}%
,4}}+\epsilon \alpha ^{-{\frac{1}{2}}-3\left( {\frac{1}{4}}-{\frac{1}{p}}%
\right) }\Vert A^{\prime }p^{\varepsilon }\Vert _{L_{t}^{2}W_{x}^{-{2},4}} \\
& \lesssim T^{{\frac{1}{4}}}\epsilon ^{{\frac{1}{8}}}\alpha ^{-{\frac{3}{2}}%
-3\left( {\frac{1}{4}}-{\frac{1}{p}}\right) }\Vert \epsilon ^{{\frac{7}{8}}%
}p_{t}^{\varepsilon }\Vert _{L_{t}^{4}W_{x}^{-{3},4}}+T^{{\frac{1}{4}}%
}\epsilon ^{{\frac{5}{8}}}\alpha ^{-{\frac{1}{2}}-3\left( {\frac{1}{4}}-{%
\frac{1}{p}}\right) }\Vert \epsilon ^{{\frac{3}{8}}}p^{\varepsilon }\Vert
_{L_{t}^{4}W_{x}^{-{2},4}}.
\end{align*}%
Now summing up the estimates for $J_{1}$ and $J_{2}$ and using Corollary \ref%
{cor_unif est} and Theorem \ref{thm_est pressure} we find that for any $%
4\leq p<6$, 
\begin{equation*}
\Vert {\mathbb{Q}}u^{\varepsilon }\Vert _{L_{t}^{2}L_{x}^{p}}\lesssim \alpha
^{1-3\left( {\frac{1}{2}}-{\frac{1}{p}}\right) }+\epsilon ^{{\frac{1}{8}}%
}\alpha ^{-{\frac{3}{2}}-3\left( {\frac{1}{4}}-{\frac{1}{p}}\right)
}+\epsilon ^{{\frac{5}{8}}}\alpha ^{-{\frac{1}{2}}-3\left( {\frac{1}{4}}-{%
\frac{1}{p}}\right) }.
\end{equation*}%
Therefore when choosing, e.g., 
\begin{equation*}
\alpha =\varepsilon ^{{\frac{1}{14}}},
\end{equation*}%
the above estimate becomes 
\begin{equation*}
\Vert {\mathbb{Q}}u^{\varepsilon }\Vert _{L_{t}^{2}L_{x}^{p}}\lesssim
\varepsilon ^{{\frac{6-p}{28p}}}+\varepsilon ^{{\frac{6+15p}{28p}}}\lesssim
\varepsilon ^{{\frac{6-p}{28p}}},\quad {\text{ for any }}\ 4\leq p<6,
\end{equation*}%
which implies (\ref{convQ}).
\end{proof}

\subsection{Strong convergence of ${\mathbb{P}}u^{\protect\varepsilon }$}

Let us first recall the celebrated Aubin-Lions lemma \cite{Au63,jllions1969}.

\begin{lemma}
\label{lem_AL} Let $X_0,X$ and $X_1$ be Banach spaces with $X_0\subset
X\subset X_1$. Suppose that $X_0$ is compactly embedded in $X$ and that $X$
is continuously embedded in $X_1$. Suppose also that $X_0$ and $X_1$ are
reflexive. For $1<p,q<\infty $, let 
\begin{equation*}
W:=\left \{u\in L^q([0,T];X_0):\ {\frac{du}{dt}}\in L^q([0,T];X_1)\right \}.
\end{equation*}
Then the embedding of $W$ into $L^p([0,T];X)$ is compact.
\end{lemma}

Next we will apply the above lemma to establish the strong compactness of
the divergence-free part of the velocity field ${\mathbb{P}}u^{\varepsilon }$%
.

\begin{lemma}
\label{lem_conv P} Let $(u^{\varepsilon },p^{\varepsilon })$ be the solution
of the Cauchy problem to system (\ref{ACm}) with initial data $%
(u_{0}^{\varepsilon },p_{0}^{\varepsilon })$ satisfying (\ref{initialcond}).
Assume also that $\varepsilon (t)$ satisfies (\ref{assumponA}). Then ${%
\mathbb{P}}u^{\varepsilon }$ is pre-compact in $L^{2}([0,T];L_{{\text{loc}}%
}^{2}({\mathbb{R}}^{3}))$.
\end{lemma}

\begin{proof}
We follow the standard idea in treating the NS equation to show that 
\begin{equation}
{\mathbb{P}}u_{t}^{\varepsilon }\ {\text{ is uniformly bounded in }}\ L^{{%
\frac{4}{3}}}([0,T];H^{-1}({\mathbb{R}}^{3})).  \label{u_t}
\end{equation}%
To this end, we apply ${\mathbb{P}}$ to the first equation in (\ref{ACm}) to
obtain 
\begin{equation*}
{\mathbb{P}}u_{t}^{\varepsilon }=\Delta ({\mathbb{P}}u^{\varepsilon })-{%
\mathbb{P}}\left[ (u^{\varepsilon }\cdot \nabla )u^{\varepsilon }\right] -{%
\mathbb{P}}\left[ {\frac{1}{2}}(\nabla \cdot u^{\varepsilon })u^{\varepsilon
}\right] .
\end{equation*}%
From Theorem \ref{thm_energy_m} we know that $u^{\varepsilon }$ is uniformly
bounded in $L^{2}([0,T];H^{1}({\mathbb{R}}^{3}))$, and hence $\Delta ({%
\mathbb{P}}u^{\varepsilon })$ is uniformly bounded in $L^{2}([0,T];H^{-1}({%
\mathbb{R}}^{3}))$. The estimates for the second and the third terms on the
right-hand side of the above equation are quite similar. So we only consider
the second term. From \cite[Lemma2.1]{temam1983book} we know that 
\begin{equation*}
\Vert (u^{\varepsilon }\cdot \nabla )u^{\varepsilon }\Vert _{H^{-1}}\leq
\Vert u^{\varepsilon }\Vert _{L^{2}}^{{\frac{1}{2}}}\Vert u^{\varepsilon
}\Vert _{H^{1}}^{{\frac{3}{2}}}.
\end{equation*}%
Therefore 
\begin{equation*}
\Vert (u^{\varepsilon }\cdot \nabla )u^{\varepsilon }\Vert _{L_{t}^{{\frac{4%
}{3}}}H_{x}^{-1}}\leq \Vert u^{\varepsilon }\Vert _{L_{t}^{\infty
}L_{x}^{2}}^{{\frac{1}{2}}}\Vert u^{\varepsilon }\Vert
_{L_{t}^{2}H_{x}^{1}}^{{\frac{3}{2}}},
\end{equation*}%
which implies (\ref{u_t}), and hence proves the lemma.
\end{proof}

\subsection{Convergence theorem}

We are now in a position to state and prove the main theorem of this section.

\begin{theorem}
\label{thm_main conv} Let $(u^{\varepsilon },p^{\varepsilon })$ be the
solution of the Cauchy problem to system (\ref{ACm}) with initial data $%
(u_{0}^{\varepsilon },p_{0}^{\varepsilon })$ satisfying (\ref{initialcond}).
Assume also that $\varepsilon (t)$ satisfies (\ref{assumponepsilon}). Then
it holds that

\begin{enumerate}
\item[(1)] there exists a $u\in L^{\infty }([0,T];L^2({\mathbb{R}}^3))\cap
L^2([0,T];\dot{H}^1({\mathbb{R}}^3))$ such that 
\begin{equation*}
u^{\varepsilon }\rightharpoonup u\quad {\text {weakly in }}\ L^2([0,T];\dot{H%
}^1({\mathbb{R}}^3)).
\end{equation*}

\item[(2)] the divergence-free part and the gradient part of $u^{\varepsilon
}$ satisfy 
\begin{equation*}
\begin{split}
&{\mathbb{P}}u^{\varepsilon }\to {\mathbb{P}}u=u\quad {\text { strongly in }}%
L^2([0,T];L_{{\text {loc}}}^2({\mathbb{R}}^3)); \\
&{\mathbb{Q}}u^{\varepsilon }\to 0\quad {\text { strongly in }}L^2([0,T];L^p(%
{\mathbb{R}}^3)),\ {\text { for any }}4\le p<6.
\end{split}%
\end{equation*}

\item[(3)] the pressure $p^{\varepsilon }$ will converge in the sense of
distribution. Indeed, 
\begin{equation*}
p^{\varepsilon }\to p=\Delta ^{-1}\nabla \cdot [(u\cdot \nabla )u]\quad {%
\text {in }}\ {\mathcal{D}}^{\prime }.
\end{equation*}
\end{enumerate}

Moreover, $u={\mathbb{P}}u$ is a Leray weak solution to the incompressible
NS equation 
\begin{equation*}
{\mathbb{P}}\left[ u_{t}-\Delta u+(u\cdot \nabla )u\right] =0\quad {\text{in 
}}\ {\mathcal{D}}^{\prime },
\end{equation*}%
and the energy inequality (\ref{energyinequ}) holds.
\end{theorem}

\begin{proof}
It is easily seen that (1) follows from Theorem \ref{thm_energy_m} and
Corollary \ref{cor_unif est}, and (2) follows from Lemmas \ref{lem_conv Q}
and \ref{lem_conv P}. The proof of (3) and the energy inequality follows the
same way as in the proof of \cite[Theorem 3.3]{donatelli2006whole}, so we omit it here. 
\bibliographystyle{plain}
\bibliography{ACreference}
\end{proof}

\section{Numerical Tests of the New Model}

To test the stability and accuracy of the new model, we perform numerical
tests of the variable timestep algorithm. The tests employ the finite
element method to discretize space, with Taylor-Hood ($\mathbb{P}_{2}/%
\mathbb{P}_{1}$) elements, \cite{G89}. The meshes used for both tests are
generated using a Delaunay triangulation. Finally, the software package
FEniCS is used for both experiments \cite{Al15}.

\subsection{Test 1: Oscillating $\protect\varepsilon(t)$}

We first apply the method to a three-\newline
dimensional offset cylinder problem. Let $\Omega
_{1}=\{(x,y,z):x^{2}+y^{2}<1,0<z<2\}$ and $\Omega
_{2}=\{(x,y,z):(x-.5)^{2}+y^{2}\leq .01,0\leq z\leq 2\}$ be cylinders of
radii 1 and .1 and height 2, respectively. Let then $\Omega =\Omega
_{1}\setminus \Omega _{2}$. Both cylinders and the top and bottom surfaces
are fixed, so no-slip boundary conditions are imposed. A rotational body
force $f$ is imposed, where 
\begin{equation*}
f(x;t):=(-4y(1-x^{2}-y^{2}),4x(1-x^{2}-y^{2}),0)^{T}.
\end{equation*}%
For initial conditions, we let $u(x;0),p(x;0)$ be the solutions to a
stationary Stokes solve at $t=0$. This does not yield a fully developed
initial condition so damped pressure oscillations\ at startup are expected
and observed. For this test, we let $\nu =.001$ and the final time $T=5$. We
let $\varepsilon _{n}=k_{n}$, where $k_{n}$ changes according the function 
\begin{equation*}
\varepsilon(t_n) = k(t_{n}):=%
\begin{cases}
.01 & 0\leq n\leq 10 \\ 
.01+.002\sin {(10t_{n})} & n>10.%
\end{cases}%
\end{equation*}%
The first plots in Figure \ref{fig:norms} below track the velocity and
pressure $L^{2}$\ norms over the duration of the simulation. After an
initial spike (typical of artificial compression methods with poorly
initialized pressures), the velocity and pressure stabilize. The vertical
axes of $||u_{h}||$ and $||p_{h}||$ are on a logarithmic scale. The variable 
$\varepsilon $, velocity, and pressure are all clearly stable. 
\begin{figure}[tbp]
\centering
\begin{subfigure}{.49\linewidth}
   \centering
   \includegraphics[width = 1\linewidth]{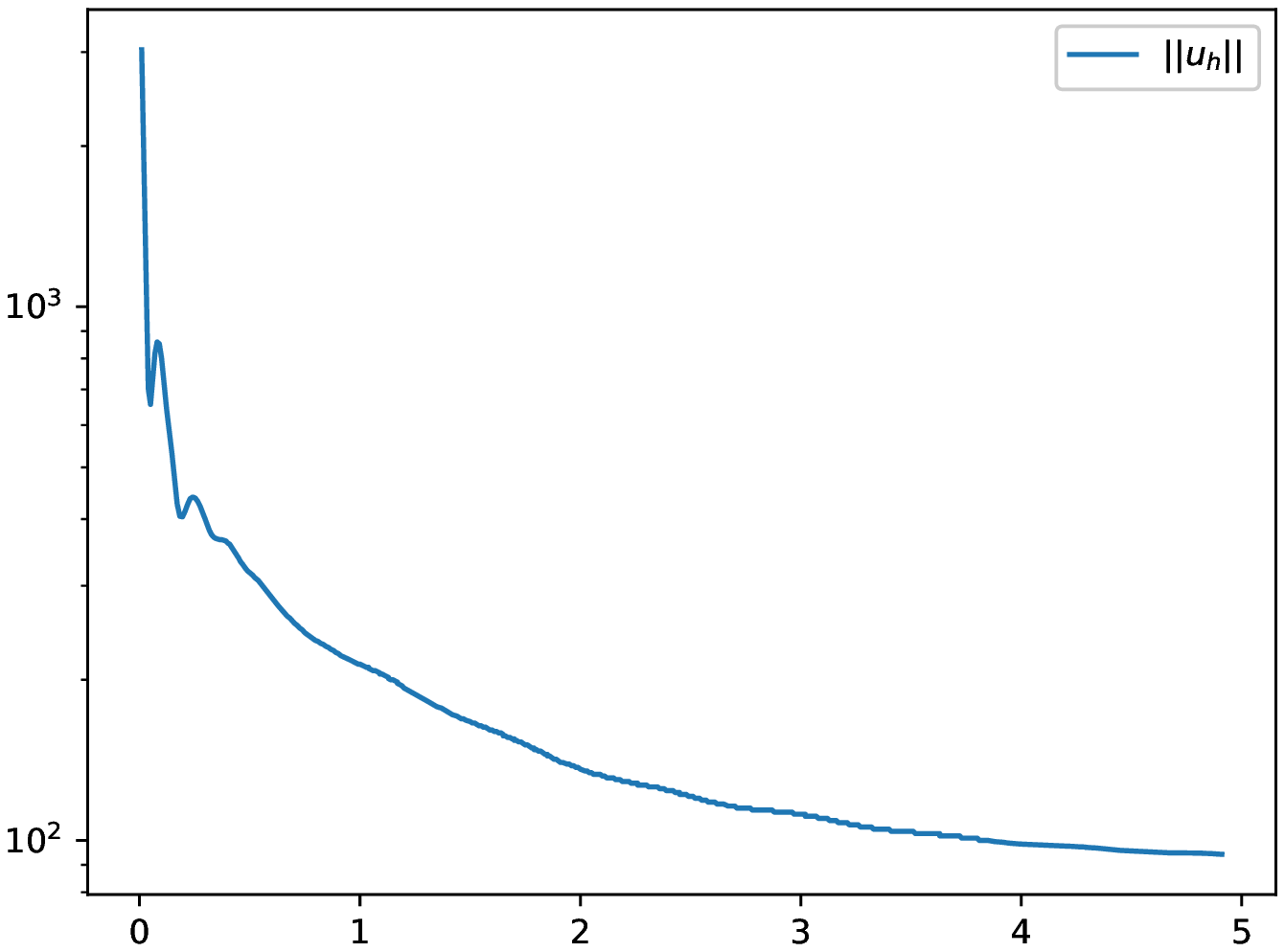}
   \caption{}
\end{subfigure}
\centering
\begin{subfigure}{.49\textwidth}
   \centering
   \includegraphics[width = 1\linewidth]{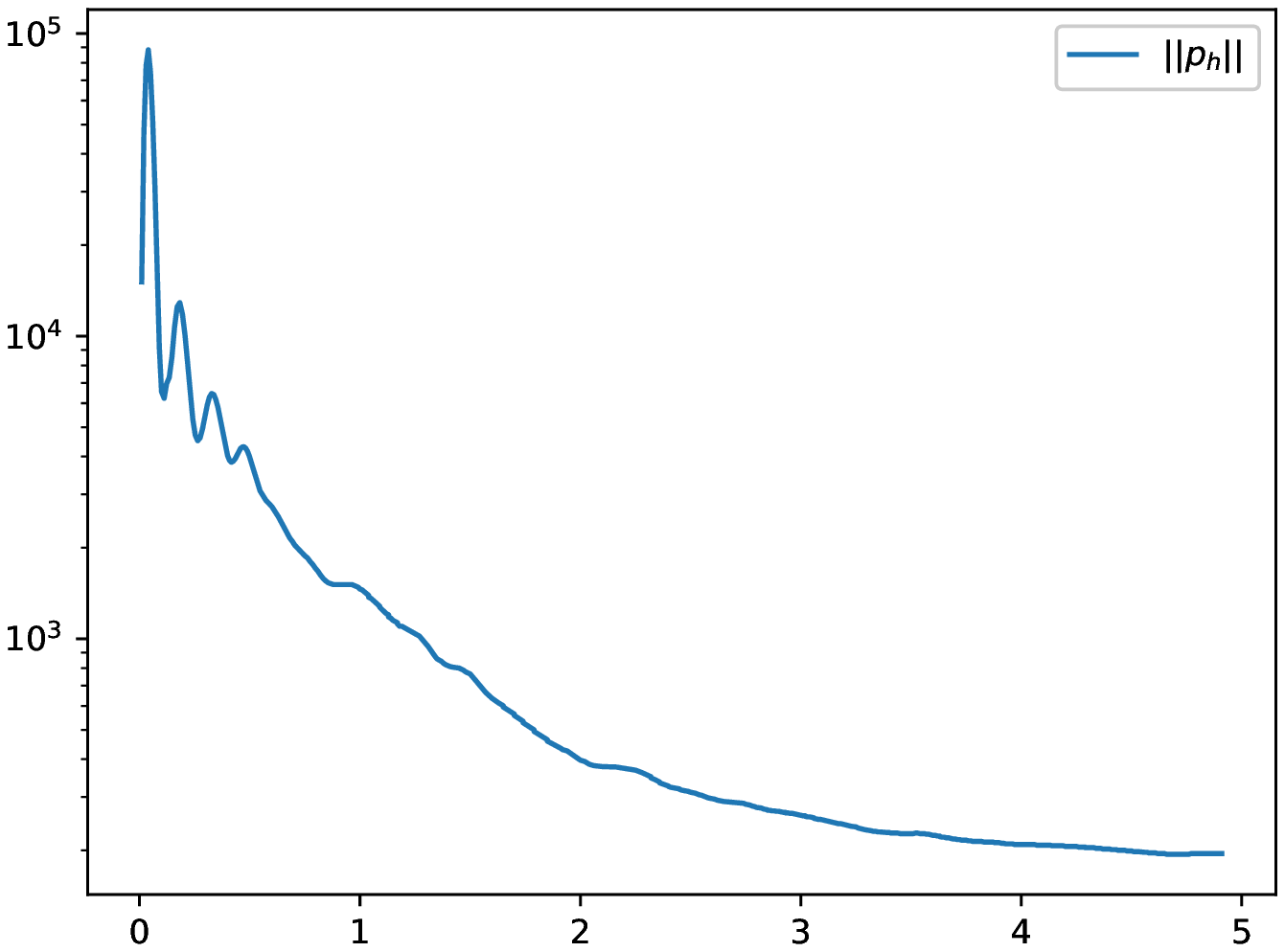}
   \caption{}
\end{subfigure}
\caption{Velocity and pressure norms over time $t$.}
\label{fig:norms}
\end{figure}
In Figure \ref{fig:cylinders}, we give plots of velocity magnitude at times $%
t=1,2,3,4$ on $\Omega $ at five cross-sections of $\Omega $. 
\begin{figure}[tbp]
\centering
\begin{subfigure}{.49\linewidth}
   \centering
   \includegraphics[width = \linewidth]{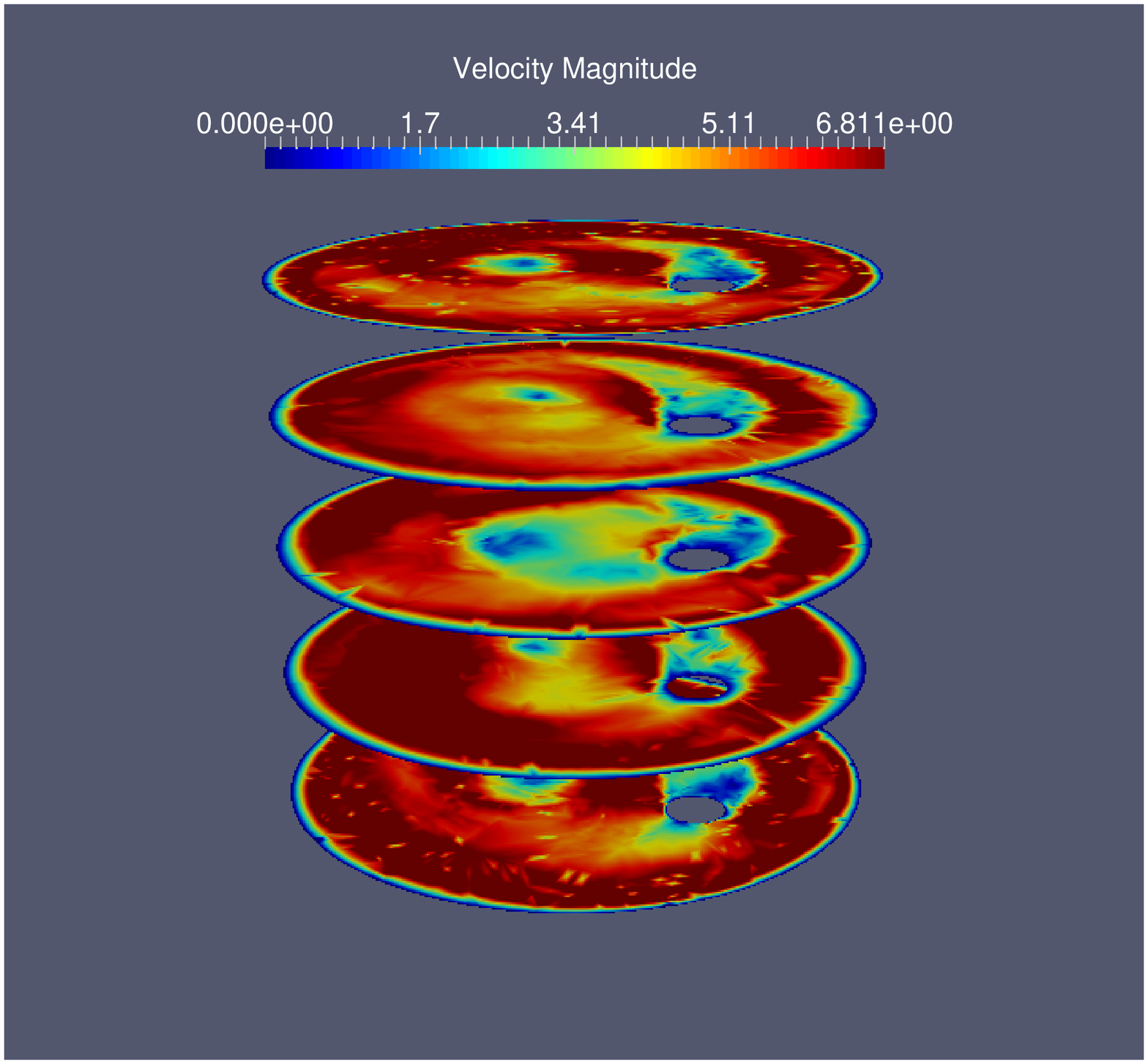}
   \caption{$t=1$}
\end{subfigure}
\centering
\begin{subfigure}{.49\textwidth}
   \centering
   \includegraphics[width = \linewidth]{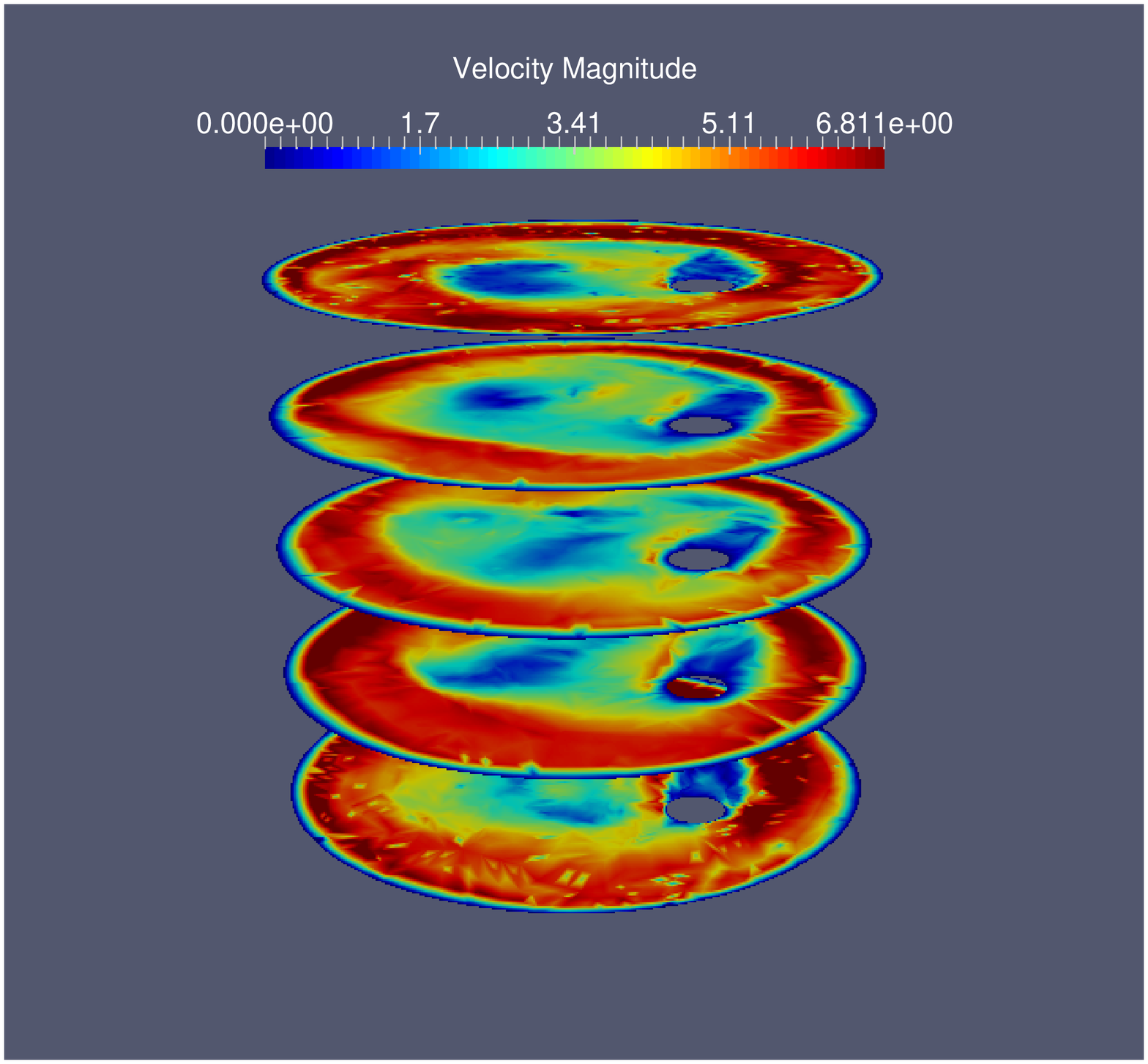}
   \caption{$t=2$}
\end{subfigure}
\end{figure}
\begin{figure}[tbp]
\begin{subfigure}{.49\linewidth}
   \centering
   \includegraphics[width = \linewidth]{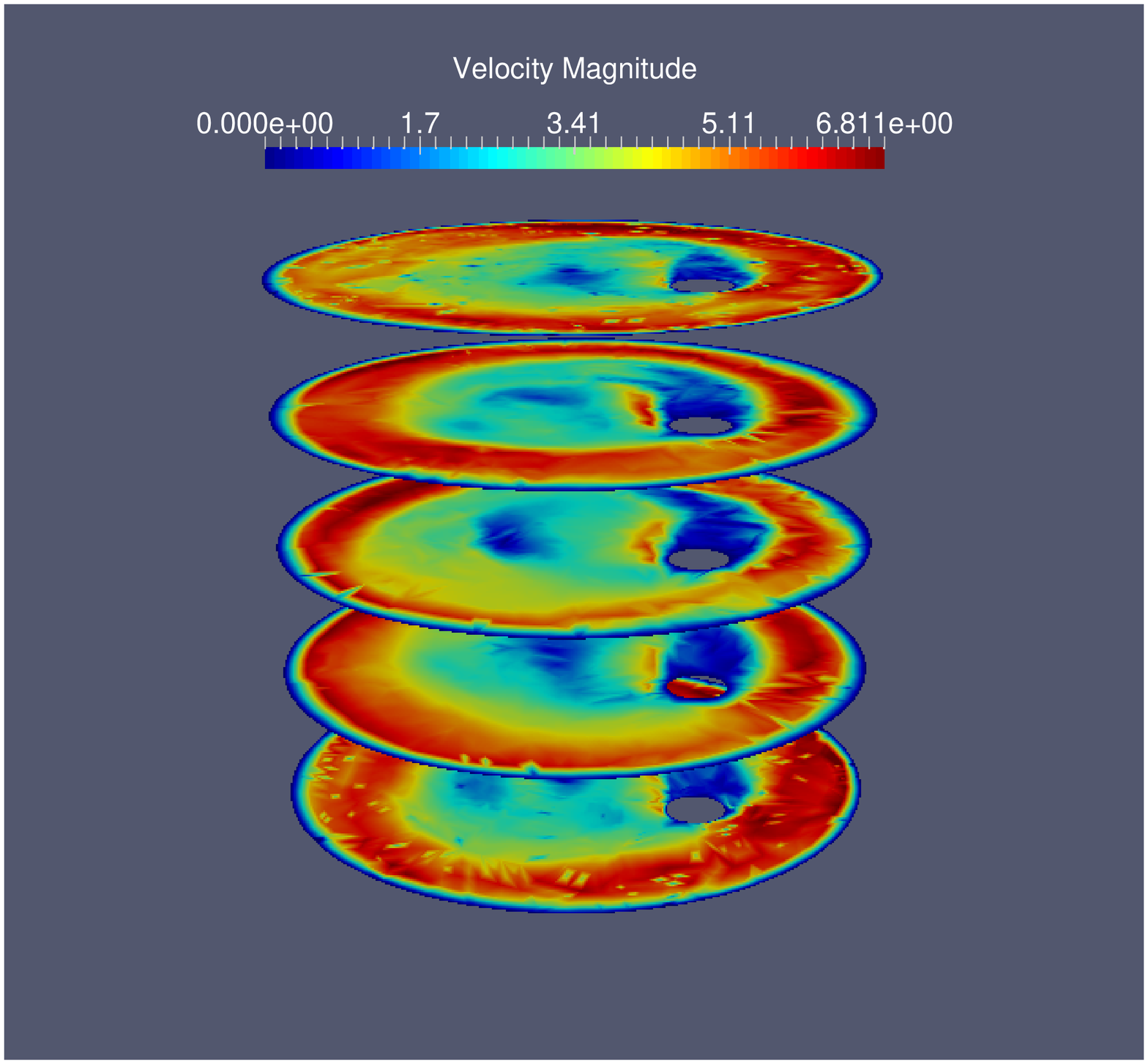}
   \caption{$t=3$}
\end{subfigure}
\centering
\begin{subfigure}{.49\textwidth}
   \centering
   \includegraphics[width = \linewidth]{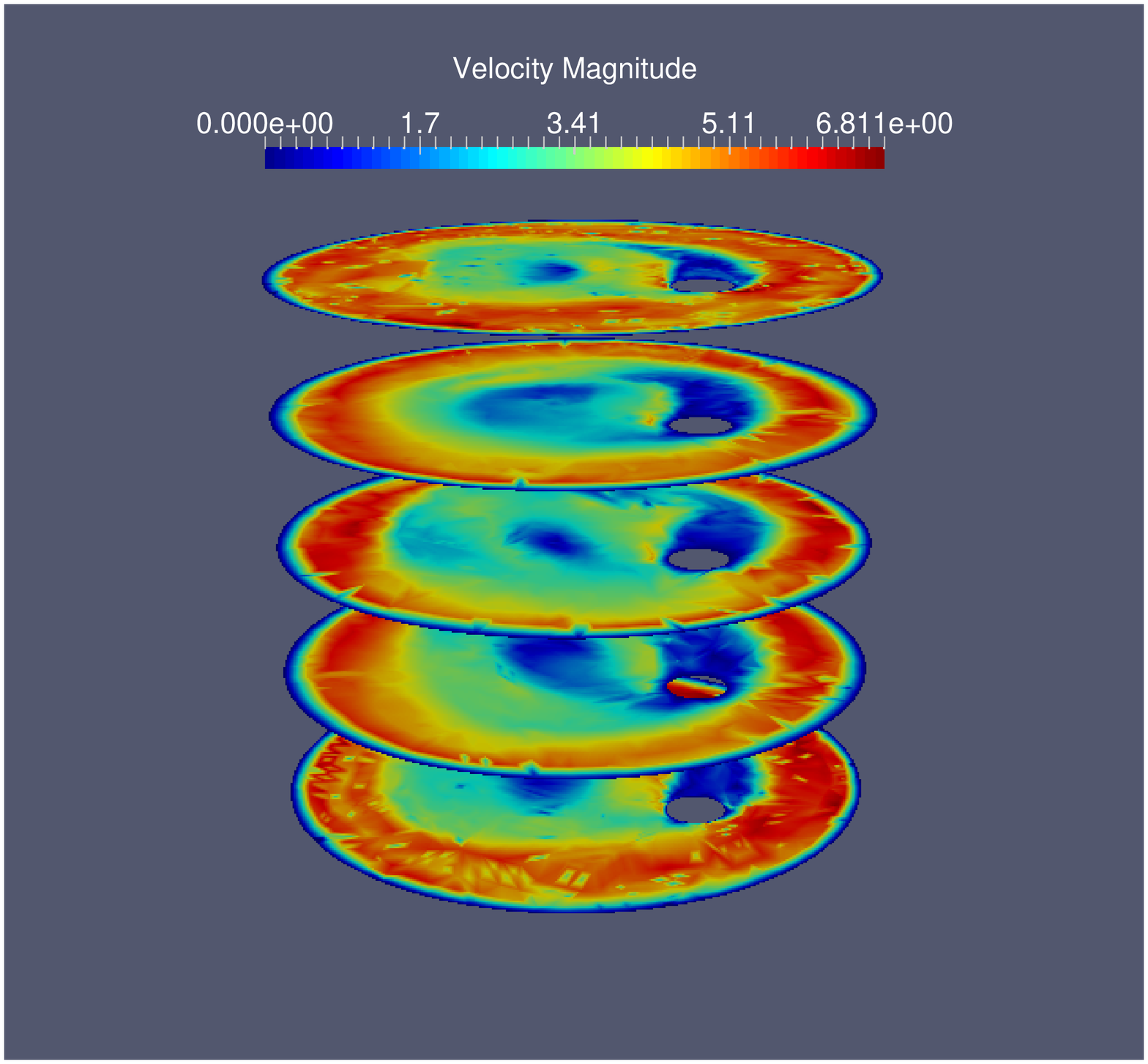}
   \caption{$t=4$}
\end{subfigure}
\caption{Velocity magnitude at different $t$.}
\label{fig:cylinders}
\end{figure}

\subsection{Test 2: Adaptive, Variable $\protect\varepsilon(t)$}

The next test investigates self-adaptive variation of $\varepsilon _{n}$ and
the resulting accuracy. We now consider a two-dimensional flow over $\Omega
=~]0,1[^{2}$ with the exact solution 
\begin{gather*}
u(x,y;t):=\sin (t)(\sin (2\pi x)\sin ^{2}(2\pi x),\sin (2\pi x)\sin
^{2}(2\pi y))^{T}, \\
p(x,y;t):=\cos (t)\cos (\pi x)\sin (\pi y)
\end{gather*}%
and corresponding body force $f$. We let $\nu =1$, the final time $T=1$, $%
\varepsilon _{n}=k_{n}$, and $k_{0}=.001$. To adapt the timestep (and
generate $k_{n}$), we employ a halving-and-doubling technique using $%
||\nabla \cdot {u_{h}}||$ as the estimator. We let the tolerance interval be 
$(.001,.01)$ (If $||\nabla \cdot {u_{h}}||<0.001$, $k_{n}$ and $\varepsilon
_{n}$ are doubled, while if $||\nabla \cdot {u_{h}}||>0.01$, the two are
halved and the step is repeated). This procedure does not control the local
truncation error, only the violation of incompressibility.

The plots in Figure \ref{fig:test2} show the velocity and pressure errors,
as well as the fluctuation of $k_n$ and $\nabla\cdot{u}$, over time. We see
that the errors of both the velocity and pressure fluctuate with changes in
the timestep, as does the divergence. 
\begin{figure}[]
\centering
\begin{subfigure}{.49\linewidth}
   \centering
   \includegraphics[width = 1\linewidth]{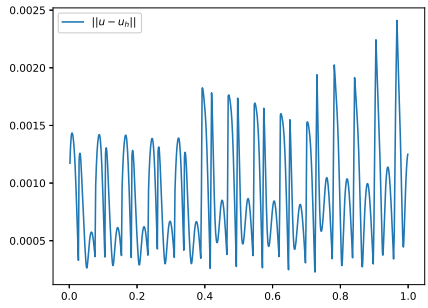}
   \caption{Velocity error}
\end{subfigure}
\centering
\begin{subfigure}{.49\textwidth}
   \centering
   \includegraphics[width = 1\linewidth]{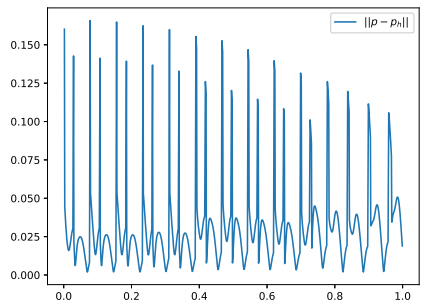}
   \caption{Pressure error}
\end{subfigure}
\begin{subfigure}{.49\linewidth}
   \centering
   \includegraphics[width = 1\linewidth]{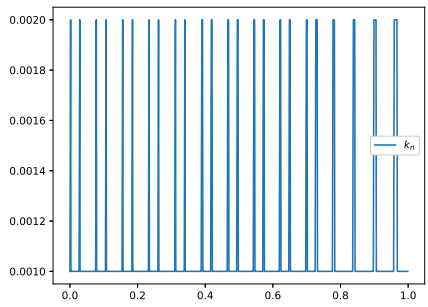}
   \caption{Timestep evolution}
\end{subfigure}
\centering
\begin{subfigure}{.49\textwidth}
   \centering
   \includegraphics[width = 1\linewidth]{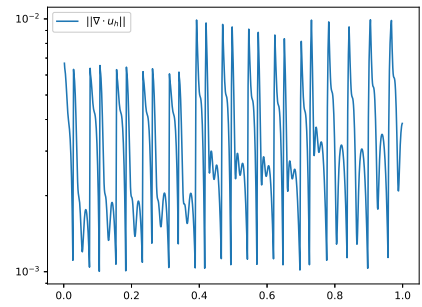}
   \caption{Divergence evolution}
\end{subfigure}
\caption{Accuracy and adaptability results.}
\label{fig:test2}
\end{figure}
Figure \ref{fig:test2}(a) shows that the velocity error is reasonable but
does grow (slowly), consistent with separation of trajectories of the
Navier-Stokes equations. Figure \ref{fig:test2}(d) shows $||\nabla \cdot {%
u_{h}}||$ is controlled. Figure \ref{fig:test2}(b) shows the pressure error
actually decreases. Figure \ref{fig:test2}(c) shows that the evolution of $%
k_{n}$, and therefore $\varepsilon _{n}$, is not as smooth as required by
condition (\ref{assumponepsilon}). Nevertheless, the simulation produced
approximations of reasonable accuracy.

\section{Conclusions and future prospects}

Slightly compressible fluids models provide a basis for challenging
numerical simulations. Efficiency and especially time accuracy in such
simulations require variable timestep and thus variable $\varepsilon
=\varepsilon (t)$. Variable $\varepsilon $ is beyond existing mathematical
foundations for slightly compressible models. The method and associated
continuum model considered herein is modified from the standard one for
variable $\varepsilon $, has been proven to be stable and converge to a weak
solution of the incompressible Navier-Stokes equations as $\varepsilon
(t)\rightarrow 0$, $\varepsilon _{t}(t)\rightarrow 0$ and $\varepsilon
_{tt}(t)\rightarrow 0$ provided $\varepsilon _{t}(t),\varepsilon
_{tt}(t)\leq C\varepsilon (t)^{1+\delta }$. The analysis of the long time
stability of the standard method and model for variable $\varepsilon
=\varepsilon (t)$ is an open problem with no clear entry point for its
analysis (Section 2). The fluctuation condition $\varepsilon
_{t}(t),\varepsilon _{tt}(t)\leq C\varepsilon (t)^{1+\delta }$ we require is
too strong for practical computation. Proving convergence to an NSE weak
solution under a relaxation of the condition is an important open problem. 
%(if e.g. $\varepsilon_{n+1} = 2\varepsilon_n )?n+1??n?tn+1=1?tn+1?n=?tn?tn+1=2?0.
Preliminary numerical tests in Section 5 with halving and doubling (which
does not satisfy the condition) suggest that the condition $\varepsilon
_{t}(t),\varepsilon _{tt}(t)\leq C\varepsilon (t)^{1+\delta }$ should be
improvable. Other open questions include convergence of flow quantities
(e.g., vorticity, lift, drag, energy dissipation rates, $Q$-criterion values
and so on) to their incompressible values as $\varepsilon (t),\varepsilon
_{t}(t)\cdot \cdot \cdot \rightarrow 0$, derivation of the rates of
convergence for strong solutions and extension of the analysis herein.

\end{document}